\documentclass{amsart}

\usepackage{amsrefs}
\usepackage{euscript}
\usepackage{epsf,color}
\usepackage{epsfig}
\usepackage{amsmath}
\usepackage{amscd}
\usepackage{amssymb}
\usepackage{enumerate}

\usepackage{amsfonts}
\usepackage{graphicx}
\usepackage[all]{xy}
\usepackage[margin=1.5in]{geometry}
\usepackage{setspace}

\numberwithin{equation}{section}

\newcommand{\bC}{{\mathbb C}}
\newcommand{\bP}{{\mathbb P}}
\newcommand{\bR}{{\mathbb R}}
\newcommand{\bZ}{{\mathbb Z}}

\newcommand{\CP}{\bC \bP}
\newcommand{\cF}{\mathcal F}

\newcommand{\vN}[1][\!]{\vec{N}^{\, #1}}
\newcommand{\vx}[1][\!]{\vec{x}^{\, #1}}
\newcommand{\ro}{{\mathrm o}}
\newcommand{\vu}[1][\!]{\vec{u}^{\, #1}}
\newcommand{\vv}[1][\!]{\vec{v}^{\, #1}}
\newcommand{\vepsilon}[1][\!]{\vec{\epsilon}^{\, #1}}

\renewcommand{\P}{{\mathrm P}}
\newcommand{\F}{{\mathrm F}}

\newcommand{\sH}{{\EuScript H}}
\newcommand{\sJ}{{\EuScript J}}
\newcommand{\bfD}{\mathbf{D}}

\newcommand{\bfp}{\mathbf{p}}

\newcommand{\coker}{\operatorname{coker}}
\newcommand{\Map}{\operatorname{Map}}
\newcommand{\Ob}{\operatorname{Ob}}
\newcommand{\ev}{\operatorname{ev}}
\newcommand{\id}{\operatorname{id}}
\newcommand{\Disc}{{\mathcal R}}
\newcommand{\Discbar}{\overline{\Disc}}
\newcommand{\Disct}{\tilde{\mathcal R}}
\newcommand{\Moore}{\EuScript P}
\newcommand{\Pil}{\mathcal H}
\newcommand{\Pilbar}{\overline{\Pil}}
\newcommand{\Morse}{{\mathcal T}}
\newcommand{\Cobord}{\mathcal C}
\newcommand{\Cobordbar}{\overline{\Cobord}}

\newcommand{\Wrap}{\EuScript{W}}

\newcommand{\Chord}{\EuScript X}
\newcommand{\Tw}{\mathrm{Tw}}
\newcommand{\Comp}{\mathrm{T}}
\newcommand{\Hom}{\operatorname{Hom}}

\newcommand{\Action}{\mathcal A}
\newcommand{\LAction}{\mathcal S}
\newcommand{\Pin}{Pin}
\newcommand{\Spin}{Spin}
\newcommand{\TQ}{T^{*} Q}
\newcommand{\Tq}{T^{*}_{q} Q}

\newcommand{\preG}{\operatorname{preG}}
\newcommand{\G}{\operatorname{G}}

\newcommand{\tr}{\mathrm{tran}}

\renewcommand{\dbar}{\overline{\partial}}
\def\co{\colon\thinspace}

\newtheorem{thm}{Theorem}[section]
\newtheorem{cor}[thm]{Corollary}
\newtheorem{lem}[thm]{Lemma}
\newtheorem{prop}[thm]{Proposition}
\newtheorem{defin}[thm]{Definition}
\newtheorem{def-lem}[thm]{Definition-Lemma}
\newtheorem{conj}[thm]{Conjecture}

\theoremstyle{remark}

\newtheorem{rem}[thm]{Remark}

\newcommand{\superscript}[1]{\ensuremath{^{\textrm{#1}}} }

\renewcommand{\th}[0]{\superscript{th}}
\newcommand{\st}[0]{\superscript{st}}

\newcommand{\noproof}{
\begin{flushright}
\qedsymbol
\end{flushright}}

\newcommand{\comment}[1]{}

\title[On the wrapped Fukaya category and based loops]{On the wrapped 
Fukaya category and based loops}

\author[M.~Abouzaid]{Mohammed Abouzaid} \date{October 3, 2010}
\thanks{ This research was conducted during the period the author served as a Clay Research Fellow. }

\begin{document}

\begin{abstract}
Given an exact relatively $\Pin$ Lagrangian embedding $Q \subset M$, we construct an $A_{\infty}$ restriction functor from the wrapped Fukaya category of $M$ to the category of modules on the differential graded algebra of  chains over the based loop space of $Q$.  If $M$ is the cotangent bundle of $Q$, this functor induces an $A_{\infty}$ equivalence between the wrapped Floer cohomology of a cotangent fibre and the chains over the based loop space of $Q$, extending a result proved by Abbondandolo and Schwarz at the level of homology.
\end{abstract}

\maketitle
\setcounter{tocdepth}{1}
\tableofcontents

\section{Introduction}
It has been known for a long time that the Floer theoretic invariants of cotangent bundles should be expressible in terms of classical invariants of the base.  The prototypical such result is Floer's proof  in \cite{floer-lagrangian} that the Lagrangian Floer cohomology groups of the zero section in a cotangent bundle agree with its ordinary cohomology groups.   Closer to the subject of this paper, Abbondandolo and Schwarz  proved in \cite{AS} that the wrapped Floer cohomology of a cotangent fibre is isomorphic to the homology of the based loop space.

The study of Fukaya categories in the setting of homological mirror symmetry as well as some of its applications to Lagrangian embeddings (see \cite{FSS}) requires understanding such Floer theoretic invariants at the chain level.  Building upon the results in this paper, we shall prove in \cite{fibre-generates} that the wrapped Fukaya category of a cotangent bundle is generated by a fibre, which makes it the most important object to study from the categorical point of view.  In Section \ref{sec:cotangent} we explain the proof of the following result.
\begin{thm} \label{thm:main}
If $Q$ is a closed smooth manifold, there exists an $A_{\infty}$ equivalence
\begin{equation}  \label{eq:isomorphism_floer_loops}  CW^{*}_{b}(T^*_{q}Q) \to C_{-*}(\Omega_{q}Q)  \end{equation}
between  the homology of the space of loops on $Q$ based at $q$ and the Floer cohomology of the cotangent fibre at $q$ taken as an object of the wrapped Fukaya category of $\TQ$ with background class $b \in H^{*}(\TQ, \bZ_{2})$ given by the pullback of $w_{2}( Q) \in H^*(Q, \bZ_{2})$.
\end{thm}
\begin{rem} 
 In Appendix \ref{sec:orientations},  we shall use a hybrid of the methods appearing in \cite{FOOO} and \cite{seidel-book} in order to define the coherent orientations of moduli spaces of holomorphic discs needed to prove this result.   In Lemma \ref{lem:twisted_to_untwisted_sign_difference}, we shall prove that the contribution of each holomorphic disc to an operation on Floer cohomology for background class $b$ differs from the corresponding contribution for the trivial background class by a sign equal to the intersection number with an appropriate cycle Poincar\'e dual to $b$.  This proves that if we consider  Floer cochains for the trivial background class, the count of holomorphic curves that we define also produces an $A_{\infty}$ equivalence
 \begin{equation}
   \label{eq:iso_twisted_loops}
   CW^{*}(T^*_{q}Q) \to C_{-*}(\Omega_{q}Q; \kappa) 
 \end{equation}
where $\kappa$ is the $\bZ$-local system on the based loop space which is  uniquely determined up to isomorphism by the property that the monodromy around a loop in $\Omega_{q} Q$ is given by the evaluation of $w_{2}(Q)$ on the corresponding torus.
\end{rem}
\begin{rem}
 Theorem \ref{thm:main} relies on the fact that a certain map we shall construct inverts an isomorphism constructed by Abbondandolo and Schwarz in \cite{AS0}.  In Summer 2009, the author was informed by Schwarz that he, together with Abbondandolo, can prove that these maps are indeed inverses and planned to write it out in an upcoming paper, leading the author to write down a proof that the maps are right inverses (the proof is sketched in Section \ref{sec:cotangent}).  An interesting analytic problem seems to arise when trying to construct by hand the homotopy that would independently prove that the maps are also left inverses, but this is of course unnecessary as the right inverse to an isomorphism is also a left inverse.
\end{rem}
\begin{rem}\label{rem:sign_error}
The original version of this paper claimed a result for general $Q$, but in fact assumed it to be $\Spin$ in the proof, and did not specify that, in the non-$\Pin$ case, the cotangent fibre should be considered as an object of a wrapped Fukaya category with a non-trivial background class  (see Section \ref{sec:at-level-objects} and Appendix \ref{sec:orientations} for a discussion of this modification of the usual Fukaya category).  In the case of non-$\Pin$ manifolds, the signed contribution of  moduli spaces of holomorphic half-strips had not been analysed.  The fact that the symplectic literature failed to account for these necessary signs was revealed, in the setting of generating functions, by work of Kragh \cite{kragh} in his generalisation of Viterbo's restriction map, and verified in a computation of the symplectic cohomology of $T^* \CP^{2}$ by Seidel in \cite{seidel-CP2} relying on Floer-theoretic methods largely independent from the ones used here.  
\end{rem}
Theorem \ref{thm:main} is a relatively straightforward application of the main construction of this paper (using as well as Abbondandolo and Schwarz's main result from \cite{AS0} at the level of homology) which concerns the wrapped Fukaya category $\Wrap(M)$ of a Liouville manifold $M$.   Recall that a Liouville manifold is an exact symplectic manifold which may be equipped with an end modelled after the positive half of the symplectisation of a contact manifold.  The wrapped Fukaya category of such a manifold has as objects exact Lagrangians modelled after Legendrians along this cylindrical end.  Such a category was defined in \cite{abouzaid-seidel}, but we shall give a different construction in Section \ref{sec:revi-wrapp-fukaya}.   In real dimension $2$, Liouville manifolds correspond to symplectic structures on punctured surfaces; a choice of Liouville form fixes a symplectomorphism from a neighbourhood of each puncture to $[1,+\infty) \times S^1$  and the collection of Lagrangians we consider agree with radial lines in these coordinates (see Figure \ref{fig:exact}).

\begin{figure} \label{fig:exact}
  \centering
  \caption{}
  \includegraphics{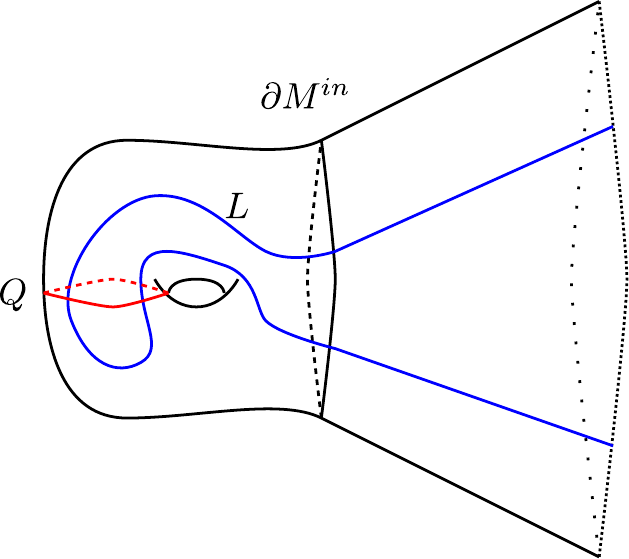}
\end{figure}

 Let us now assume that a closed manifold $Q$  embeds as an exact Lagrangian in $M$, and that there exists a class $b \in H^{2}(M, \bZ_{2})$ whose restriction to $Q$ agrees with the  second Stiefel-Whitney class $w_{2}(Q)$; we say that such a Lagrangian is \emph{relatively $\Pin$} for the background class $b$.

We associate to $Q$ a category $\Moore(Q)$ whose objects are points of $Q$ and morphisms are chains on the spaces of paths between such points.  The endomorphism algebra of any object in this category is $C_{-*}(\Omega_{q}Q)$, which can be made into a differential graded algebra by using normalised cubical chains as explained in the next section.  Given an exact Lagrangian submanifold  $L$ whose second Stiefel-Whitney class is also given by the restriction of $b$, we construct an explicit module $\cF(L)$ over the category $\Moore(Q)$, in fact a twisted complex, using the moduli spaces of holomorphic strips bounded on one side by $Q$ and on the other by $L$.    With minor technical differences coming from choices of basepoints, such a module was constructed by Barraud and Cornea in \cite{BC}. We prove that this assignment extends to an $A_{\infty}$ functor
\begin{equation}  \Wrap_{b}(M) \to \Tw(\Moore(Q)), \end{equation}
where the left hand side is the wrapped Fukaya category with respect to the background class $b$, and the right hand side is the category of twisted complexes.  Applying this functor to $M = T^*Q$ and $L= T^*_{q}Q$ yields Theorem \ref{thm:main}.  Moreover, using Theorem \ref{thm:main} we conclude
\begin{cor} \label{cor:restriction}
If $Q$ is an exact relatively $\Pin$ Lagrangian embedded in $M$, there is an $A_{\infty}$ restriction functor
\begin{equation} \Wrap_{b}(M) \to \Tw( \Wrap_{b}(T^*Q) ) .\end{equation}
\end{cor}
\begin{proof}
We may split the map \eqref{eq:isomorphism_floer_loops} to obtain an $A_{\infty}$ equivalence
\begin{equation} \Tw \Moore(Q) \to  \Tw(   CW^{*}_{b}(T^*_{q}Q) ).  \end{equation}
The left hand side may be thought of as twisted complexes on $\Wrap_{b}(T^*Q)$ which are built using only the object $ T^*_{q}Q $.  Composing this equivalence with the inclusion of the right hand side in $\Tw( \Wrap_{b} (T^*Q) )$ and the functor $\cF$, we obtain the desired result.
\end{proof}

The reader may want to compare this circuitous construction of a restriction map with the one defined in \cite{abouzaid-seidel} for an inclusion $M^{in} \subset M$ of a Liouville subdomain.  The key differences are that the construction in \cite{abouzaid-seidel} can only be performed for Lagrangians which intersect the boundary of  $M^{in}$ in a very controlled way, and that for such Lagrangians the image of the restriction functor lies in  $\Wrap(M^{in})$ (rather than twisted complexes thereon).  Corollary \ref{cor:restriction} leads us to make the following conjecture:
\begin{conj}
Any inclusion of a Liouville subdomain $M^{in} \subset M$ induces an $A_{\infty}$ restriction functor
\begin{equation} \Wrap(M) \to  \Tw( \Wrap(M^{in}) ) . \end{equation}
\end{conj}
It is important to note that this result cannot hold if we do not pass to twisted complexes on the right hand side (or derived categories if the reader prefers that language).

\subsection*{Acknowledgments}
I would like to thank Ivan Smith for the invitation to visit Cambridge University in the Winter of 2008, when we discussed wrapped Fukaya categories in various setting, and Ralph Cohen for the invitation to lecture at the March 2008 Stanford University workshop on String Topology, where I presented a preliminary version of these results.  Conversations, electronic and otherwise, with Ralph Cohen and Andrew Blumberg as well as Peter Albers and Dietmar Salamon helped, respectively on the topological and symplectic parts, to clarify the technical problems to be circumvented.   Finally, I would like to thank Thomas Kragh and Paul Seidel for discussions about the sign error mentioned in Remark \ref{rem:sign_error}.

\subsection*{Conventions}
For the parts of the paper dealing with the homological algebra of $A_{\infty}$ categories, we shall mostly use results which appear in the first part of  \cite{seidel-book}.   In particular, our sign conventions for orienting moduli spaces of holomorphic discs are the same as those appearing in \cite{seidel-book}, as well as in \cite{abouzaid-seidel}
\section{Construction of the functor at the level of objects}
\subsection{The Pontryagin category}
Let $Q$ be any path connected topological space.  Consider the topological category with objects the points of $Q$, and morphisms from  $q^0$ to $q^1$ given by the Moore path space
\begin{equation} \Omega(q^0,q^1)  \equiv  \{ \gamma \co [0,R] \to Q | \gamma(0) = q^0, \gamma(R) = q^1 \} \end{equation}
where $R$ is allowed to vary between $0$ and infinity.  The composition law is given by concatenating the domains and the maps:
\begin{align*} \Omega(q^0,q') \times  \Omega(q',q^1)  & \to  \Omega(q^0,q^1) \\
(\gamma_1, \gamma_2) & \to \gamma_1 \cdot \gamma_2 (l) \equiv \begin{cases} \gamma_1(l) & \textrm{ if } 0 \leq l \leq R_1 \\
\gamma_2( l -R_1)  & \textrm{ if } R_1 \leq l \leq R_1 + R_2 \end{cases}
 \end{align*}
where $\gamma_i$ is assumed to have domain $[0,R_i]$.

It is well known that this formula defines an associative composition of paths.   In order for this operation to induced the structure of a differential graded algebra on chains, we use normalised cubical chains throughout this paper.  Recall, for example from \cite{massey}, that a map from a cube to a topological space is said to be degenerate if it factors through the projection to a factor.  The graded abelian groups underlying the normalised chain complex are
\begin{align*}
 C_{i}(X) & = \frac{\bZ\left[ \Map([0,1]^{i}, X) )   \right] }{\bZ \left[\textrm{degenerate maps}\right]}
\end{align*}
Writing $\delta_{k,\epsilon}$ is the inclusion of the face where the $k$\th coordinate is constant and equal to $\epsilon$, we define a differential by the formula 
\begin{equation*} 
  \partial \sigma =  \sum_{k=1}^{i} \sum_{\epsilon = 0,1}   \partial_{k,\epsilon} \sigma =  \sum_{k=1}^{i} \sum_{\epsilon = 0,1}   (-1)^{k + \epsilon}  \sigma \circ \delta_{k,\epsilon}.
\end{equation*}

The key difference with the theory based on simplices is that a product of cubes is again a cube, so it is easy to define a map
\begin{equation} C_{*}(X) \times C_{*}(Y)  \to  C_{*}(X \times Y) \end{equation}
which may easily be checked to be associative in the appropriate sense.  Applying this to path spaces, we obtain a differential graded category. In order to be consistent with our sign conventions for the Fukaya category, we denote by
\begin{equation} \Moore(Q) \end{equation}
the $A_{\infty}$ category with objects points of $Q$, morphism spaces
\begin{equation} \Hom_{*}(q^0, q^1) =  C_{-*}( \Omega(q^0,q^1) ) \end{equation}
and differential and product  
\begin{align} \label{eq:signs_a_infty_cubical} 
\mu^{\P}_1 \sigma & \equiv  \partial \sigma \\
\mu_{2}^{\P}(\sigma_{2}, \sigma_{1}) & \equiv (-1)^{\deg \sigma_1} \sigma_{1} \cdot \sigma_{2}.
\end{align}

Note that the path-connectivity assumption on $Q$ implies that all objects of this category are quasi-isomorphic.  In particular, the inclusion of any object defines a fully faithful $A_{\infty}$ embedding
\begin{equation} C_{-*}(\Omega(q,q)) \to \Moore(Q) . \end{equation}

We shall need to consider an enlargement of   $\Moore(Q)$ to a triangulated $A_{\infty}$ category.  The canonical such enlargement is the triangulated closure of  the image of $\Moore(Q)$ in its category of modules under the Yoneda embedding.  In practice, it shall be convenient to use the more explicit model of twisted complexes introduced by Bondal and Kapranov in  \cite{bondal-kapranov}.  First, we enlarge $\Moore(Q)$ by allowing shifts of all objects (by arbitrary integers) and define
\begin{equation} \Hom_{*}(q^0[m_0], q^1[m_1]) \equiv \Hom_{*}(q^0,q^1)[m_1- m_0]  \end{equation}
with differential $\mu_{1}^{\P}$.  Given a triple $(q^0[m_0],q^1[m_1],q^1[m_2])$, and morphisms $\sigma_{i} \in \Hom_{*}(q^{i-1},q^{i})$  multiplication is defined via
\begin{equation} \label{eq:sign_additive_enlargement}   \mu^{\P}_{2}(\sigma_2 [m_2 -m_1]  ,\sigma_1 [m_1 -m_0]) = (-1)^{ (\deg(\sigma_2)+1)(m_1 -m_0)}   \mu^{\P}_{2} ( \sigma_2, \sigma_1) [m_2 - m_0]. \end{equation} 
\begin{defin}
A \emph{ twisted complex} consists of the following data
\begin{equation} \label{eq:twisted_complex_equation}
\begin{array}{c} 
\parbox{35em}{A finite collection of objects $\{ q^{i} \}_{i=1}^{r}$ and integers $m_i$, together with a collection of morphisms $\{ \delta_{i,j} \}_{i<j}$ of degree $1$ in $ \Hom_{*}(q^i[m_i], q^j[m_j]) $, such that} \\
\displaystyle{ \mu^{\P}_{1}\delta_{i,j} + \sum_{k} \mu^{\P}_{2}( \delta_{k,j} , \delta_{i,k}) = 0}
\end{array}
\end{equation}
\end{defin}
We write $D$ for the matrix of morphisms  $\{ \delta_{i,j} \}_{i<j}$, and $\Comp = ( \oplus  q^{i}[m_i], D)$ for such a twisted complex, and note that the equation imposed on $\delta_{i,j}$ can be conveniently encoded as
\begin{equation} \mu^{\P}_{1}( D) + \mu^{\P}_{2}( D,D) = 0 .\end{equation}
Given two such complexes $\Comp^1 = ( \oplus  q_{1}^{i}[m^1_i], D^1)$ and $\Comp^2= ( \oplus  q_{2}^{i}[m^2_i], D^2)$, we define the space of morphisms between them as a direct sum
\begin{equation} \Hom_{*}(\Comp^1, \Comp^2) \equiv \bigoplus_{i_1,i_2} C_{*}(\Omega_{q_1^{i_1}, q_{2}^{i_2}})[m^1_{i_1} - m^2_{i_2}]. \end{equation}
The differential of an element $S = \{ \sigma_{i_1,i_2} \}$ in this space is given for an elementary matrix by
\begin{equation}  \label{eq:differential_twisted} \sigma_{i_1,i_2} \mapsto \mu_1^{\P} \sigma_{i_1,i_2} +  \sum_{k < i_1} \mu^{\P}_{2} (\sigma_{i_1,i_2} , \delta^{1}_{k,i_1}) +  \sum_{i_2 <k}  \mu^{\P}_{2} (\delta^{2}_{i_2,k} , \sigma_{i_1,i_2}), \end{equation}
which can be written much more clearly in terms of matrix multiplication 
\begin{equation}\mu^{\Tw(\Moore)}_{1} S =  \mu_1^{\P} S + \mu_{2}^{\P}(S,D^{1})  + \mu_{2}^{\P}(D^{2},S). \end{equation}

Composition is also most clearly defined in terms of matrix multiplication
\begin{align*} 
\mu_{2}^{\Tw(\Moore)} (S_2, S_1) & \equiv \mu_{2}^{\P}(S_2,S_1).
 \end{align*} 

The following result is essentially due to  Bondal and Kapranov in  \cite{bondal-kapranov}:
\begin{lem}
Twisted complexes form a triangulated $A_{\infty}$ category. \noproof
\end{lem}

\subsection{Assigning a twisted complex to each Lagrangian} \label{sec:at-level-objects}
Let $M$ be a  symplectic manifold equipped with a $1$-form $\lambda$ whose differential $\omega$ is a symplectic form.  Assume the existence of a compact (codimension $0$) submanifold $M^{in}$ with boundary such that the restriction of $\lambda$ to $\partial M^{in} $ is a contact form.  We say that $M$ is a Liouville manifold if we have a decomposition
\begin{equation*}
M = M^{in} \cup_{\partial M^{in}}  [1, +\infty) \times \partial M^{in},
\end{equation*}
such that the Liouville form is given by $\lambda = r (\lambda| \partial M) $  on the infinite end $   [1, +\infty) \times \partial M^{in} $, where $r$ is the coordinate on $[1,+\infty) $. 

Instead of studying all Lagrangians in $M$, we consider  a \emph{closed} exact Lagrangian $Q \subset M$, together with a finite collection $\Ob(\Wrap_{b}(M))$ of exact properly embedded Lagrangians such that 
\begin{equation} \label{eq:boundary_Legendrian} \parbox{36em}{$\lambda$ vanishes on $L \cap \partial M^{in} \times [1,+\infty)$ if $L \in \Ob(\Wrap_{b}(M))$, and $Q$ intersects each Lagrangian in $\Ob(\Wrap_{b}(M))$ transversely.} \end{equation}
 The first condition is equivalent to the requirement that the intersection $\partial L$  of $L$ with $\partial M^{in}$ be Legendrian, and that $L$ be obtained by attaching an infinite cylindrical end to the intersection of $L$ with $M^{in}$
\begin{equation*}  L  = L^{in} \cup_{\partial L^{in}}  [1, +\infty) \times \partial L^{in}. \end{equation*} 
We shall choose a primitive $f_{L}$ for the restriction of $\lambda$ to each such Lagrangian, which by the above conditions is locally constant away from a compact set, as well as a primitive $f_{Q}$ for the restriction of $\lambda$ to $Q$.  The exactness condition excludes bubbling of holomorphic discs; a general Lagrangian Floer theory is developed by Fukaya, Oh, Ohta and Ono in \cite{FOOO}, but a wrapped version of their theory has not been worked out.  It is also likely that the theory works under weaker assumptions on the properties of $L$ at infinity, but these would require more delicate estimates on the behaviour of solutions to the perturbed $\dbar$ equations we shall study, which discourages us from pursuing this generalisation.

With the condition imposed above, we may define a $\bZ_{2}$-graded Fukaya category over a field of characteristic $2$.  To obtain $\bZ$-gradings and work over the integers, we assume that for each $L \in  \Ob(\Wrap_{b}(M))$ or for $L=Q$,  
\begin{equation} \label{eq:relative_pin_graded} \parbox{36em}{the restriction of $b$ to $L$ agrees with the second Stiefel-Whitney class $w_{2}(L)$.  Moreover,  the relative first Chern class $2 c_{1}(M,L)$ vanishes on  $ H_{2}(M,L)$.} \end{equation}
The condition on $  c_{1}(M,L) $ incorporates the requirement that $M$ have a quadratic complex volume form when equipped with any compatible almost complex structure,  and that the Maslov index on $H_{1}(L)$ vanish.    This class is well defined because the space of almost complex structures is contractible.  Since $L$ is Lagrangian, we may evaluate such a volume form  at every point $p \in L$ to obtain an element of $\bC^{*}$. The vanishing of the Maslov index on  $ H_{1}(L) $ implies that this function $L \to \bC^{*}$ may be factored through the universal cover of $ \bC^{*} $.   The choice of such a lifting is called a \emph{grading} on $L$ (see Section 12 of \cite{seidel-book}). 

On the other hand, the condition that $b$ agree, when restricted to $L$, with $w_{2}(L)$ allows us to choose a relative $\Pin$ structure on $Q$ as well as on the objects of $\Wrap_{b}(M)$.  Instead of following the approach using non-abelian cohomology (see  e.g. Section (11i) of \cite{seidel-book} and \cite{WW}), we shall adopt a variant of the method used in Chapter 9 of \cite{FOOO} for orienting moduli spaces of holomorphic curves.  Namely, we fix a triangulation of $M$ so that $Q$ and the elements of $ \Ob(\Wrap_{b}(M)) $ are subcomplexes and fix an orientable vector bundle $E_{b}$ on the $3$-skeleton of this triangulation whose second Stiefel-Whitney class is the restriction of $b$ to the $3$-skeleton. 

 Condition \eqref{eq:relative_pin_graded}, together with the fact that $E_{b}$  is orientable, implies that 
 \begin{equation}
  w_{2}( TL |L_{[3]} \oplus E_{b} | L_{[3]} ) = 0
 \end{equation}
where $L_{[3]}$ stands for the $3$-skeleton of $L$ with respect to the triangulation induced as a subcomplex of $M$.  We refer the reader to Section (11i) of \cite{seidel-book} for a discussion of the double cover of the orthogonal group called $\Pin$, and the fact that the second Stiefel-Whitney class is precisely the obstruction to the existence of $\Pin$ structures:
\begin{defin}
 A relative $\Pin$ structure on $L$ is the choice of a $\Pin$ structure on the vector bundle $ TL|L_{[3]} \oplus E_{b} | L_{[3]}  $ defined over the $3$-skeleton of $L$.
\end{defin}
\begin{defin} \label{def:brane}
A \emph{brane structure} on a Lagrangian is the choice of a grading, together with a relative $\Pin$ structure. 
\end{defin}
Let us fix once and for all such a structure on $Q$, as well as on each Lagrangian $L \in  \Ob(\Wrap_{b}(M))$, which is precisely the data needed to unambiguously assign gradings to Floer groups and signs to operations thereon.   In the next section, we shall review the construction of a category we call the wrapped Fukaya category, whose objects are $  \Ob(\Wrap_{b}(M)) $.

To each object $L$ of the wrapped Fukaya category, we assign a twisted complex over the Pontryagin category of $Q$ as follows:   We write $\{ q^{i} \}_{i=1}^{m}$  for the set of intersection points between $Q$ and $L$, which we assume are ordered by action, i.e. the difference between $f_Q$ and $f_L$.  Moreover, we write $\sJ(M)$ for the space of almost complex structures on $M$ which are compatible with the symplectic form on $M$, and such that
\begin{equation*}
  \lambda \circ J = dr
\end{equation*}
on the cylindrical end.  This is a mild version of the contact-type property often imposed on almost complex structures on symplectic manifolds with contact boundary.

Fixing a family $I_{t} \in \sJ(M)$ of such almost complex structures parametrised by $t \in [0,1]$, we consider the moduli spaces
\begin{equation} \Pil(q^i, q^j) \end{equation}
which are the quotients by the $\bR$ action of the space of solutions to the time-dependent $\dbar$-equation
\begin{equation}  \label{eq:dbar_no_X}
    du(s,t) \circ j - I_{t} \circ du(s,t) = 0
\end{equation}
with boundary conditions
\begin{equation} \label{eq:holomorphic_strip}
\left\{
\begin{aligned}
 & u: Z = \bR \times [0,1] \longrightarrow M, \\
 & u \left(\bR \times \{ 1 \}  \right) \subset  L, \\
 & u \left(\bR \times \{ 0 \}  \right) \subset  Q, \\
 &  \lim_{s \rightarrow - \infty} u(s,t) = q^j , \\ 
&  \lim_{s \rightarrow + \infty} u(s,t) = q^i
\end{aligned}
\right.
\end{equation}
These boundary conditions are conveniently summarised in Figure \ref{fig:strip}.  
\begin{figure}
  \centering
 \includegraphics{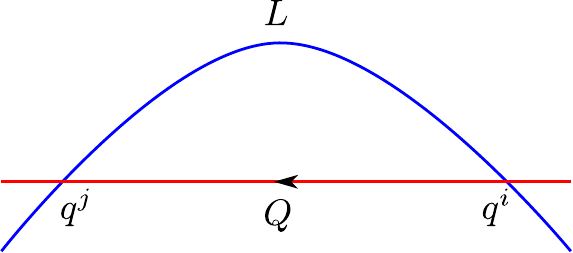}
  \caption{}
 \label{fig:strip}
\end{figure}

As the Lagrangians $Q$ and $L$ are both graded, we may assign an integer to each intersection point $q \in Q \cap L$ called the degree of $q$ and denoted $|q|$.  If we were trying to define the Floer cohomology $HF^*(L,Q)$, this would be the degree of the corresponding generator of the Floer complex.

\begin{prop} \label{prop:moduli_spaces_strips_manifold}
For a generic family $I_{t}$, the Gromov compactification of $ \Pil(q^i, q^j) $ is a topological manifold  $\Pilbar(q^i, q^j)$ of dimension $|q^i| - |q^j|$, possibly with boundary.  The boundary is stratified into topological manifolds with the closure of codimension $1$ strata given by the images of embeddings
\begin{equation} \Pilbar(q^i, q^{i_1}) \times \Pilbar(q^{i_1}, q^j) \to  \Pilbar(q^i, q^j) \end{equation}
for all possible integers $i_{1}$ between $i$ and $j$.  For each pair $i_1 < i_2$ of such integers, we have a commutative diagram 
\begin{equation} \label{eq:compatible_boundary_strata_inclusions}  \xymatrix{ \Pilbar(q^i, q^{i_1}) \times \Pilbar(q^{i_1}, q^{i_2}) \times  \Pilbar(q^{i_2}, q^{j}) \ar[d] \ar[r] &  \Pilbar(q^i, q^{i_1}) \times \Pilbar(q^{i_1}, q^{j}) \ar[d] \\
 \Pilbar(q^i, q^{i_2}) \times \Pilbar(q^{i_2}, q^{j}) \ar[r] &  \Pilbar(q^i, q^{j}) . }
 \end{equation}
\end{prop}
\begin{rem}
 In Appendix \ref{app:manifold-with-boundary} we explain how to prove the result we need using only ``standard results''.  In the appendix to \cite{BC}, Barraud and Cornea provide an alternative construction.
\end{rem}

In order to define an evaluation map from the moduli space of holomorphic discs to the space of paths on $Q$, we must fix parametrisations.  One may inductively make choices of such parametrisations using the contractibility of the space of self-homeomorphisms of the interval.  Alternatively, as suggested to the author by Janko Latschev, one may fix a metric on $Q$, and  parametrise the boundary segments of a holomorphic disc by arc length.  Indeed, it follows from elliptic regularity that each solution to \eqref{eq:holomorphic_strip} restricts on the appropriate boundary component to a  smooth map from $\bR$ to $Q$, and from Theorem A of \cite{RS}, that its derivatives decay exponentially in the $C^{\infty}$ topology at the ends.  In particular, the arc-length parametrisation of such a curve defines a continuous map from $[0,R]$ to $Q$, with $R$ the length of the image.

As a sequence of holomorphic strips converge to a broken one, the length parametrisations of the boundary also converge: this follows from the usual proof of Gromov compactness by using estimate (4.7.13) of \cite{MS} which proves that, near the region where breaking occurs, the norm of the derivatives of a family of holomorphic strips decays exponentially, so that the resulting paths on $ Q $  converge to the concatenation of two curves.  This implies that this evaluation map extends continuous to the Gromov compactification of the moduli space of discs.    Note that the direction of the direction of the path, as indicated in Figure \ref{fig:strip}, goes from $q^i$ to $q^j$.

\begin{lem} \label{lem:evaluate_discs_paths}
There exists a family of evaluation maps
\begin{equation}  \Pilbar(q^i,q^j) \to \Omega(q^i,q^j) \end{equation}
such that we have a commutative diagram
\begin{equation}\label{eq:evaluate_discs_paths_commute}   \xymatrix{ \Pilbar(q^i,q^k) \times  \Pilbar(q^k,q^j) \ar[d] \ar[r] & \Pilbar(q^i,q^j)\ar[d] \\
 \Omega(q^i,q^k) \times  \Omega(q^k,q^j) \ar[r] &  \Omega(q^i,q^j).}
  \end{equation} 
in which the top horizontal arrow is the inclusion of Equation \eqref{eq:compatible_boundary_strata_inclusions} and the bottom arrow is given by concatenation of paths.
\noproof
\end{lem}

Before passing to chains, we recall that the choice of brane structure can be used to orient these moduli spaces.  Our orientation conventions follow those of \cite{seidel-book}, and are discussed in Appendix \ref{sec:orientations} and \ref{sec:signs}
\begin{lem} \label{lem:product_orientations_strips}
The product orientation on  $\Pilbar(q^i, q^k) \times \Pilbar(q^k, q^j)$ differs from its orientation as a boundary of $\Pilbar(q^i, q^j)$ by a sign given by the parity of $|q^i| + |q^k|$. \end{lem}

This result implies that we may choose fundamental chains $[\Pilbar(q^i, q^j)]$ in the cubical chain complex $C_{*}( \Pilbar(q^i, q^k))$ inductively.  We start by picking representatives of the fundamental cycles of all components of $\bigcup \Pilbar(q^i, q^j)$ which are manifolds without boundary, compatibly with the orientations determined by our choices of brane data and the isomorphism of  Equation \eqref{eq:orientation_isomorphism_discs}.

Next, we proceed by induction, assuming that a class satisfying
\begin{equation} \label{eq:boundary_strips_breaking}  \partial [\Pilbar(q^i, q^{j})] \\ = \sum_{k} (-1)^{ |q^i|+  |q^k| }  [\Pilbar(q^i, q^k)] \times  [\Pilbar(q^k, q^{j'})]  \end{equation}
has been chosen for all moduli spaces of dimension smaller than some integer $m$.  In the induction step, we note that the associativity of the cross product and the commutativity of the diagram \eqref{eq:evaluate_discs_paths_commute} imply that
\begin{equation} \partial \left( \sum_{k} (-1)^{ |q^i|+  |q^k|  }  [\Pilbar(q^i, q^k)] \times  [\Pilbar(q^k, q^j)] \right) = 0.  \end{equation}
As this is a closed class of codimension $1$ supported on the boundary of $\Pilbar(q^i, q^j)$, we may choose a bounding class  $[\Pilbar(q^i, q^j)]$.

\begin{lem} \label{lem:definition_twisted_complex_Lag}
Each Lagrangian $ L \in \Ob(\Wrap_{b}(M))$ determines a twisted complex $\cF(L)$ in $\Tw(\Moore(Q))$ given by
\begin{equation} \left( \bigoplus_{q^i \in Q \cap L} q^i [-|q^i|], \sum_{q^i, q^j} (-1)^{|q^i| \left( |q^j| +1 \right)} [\Pilbar(q^i, q^j)] \right). \end{equation}
\end{lem}
\begin{proof}
 We must verify Equation \eqref{eq:twisted_complex_equation}, which takes the form
 \begin{equation*}
 (-1)^{|q^i| \left( |q^j| +1 \right)} \partial  [\Pilbar(q^i, q^j)]   + \sum_{k}  (-1)^* [\Pilbar(q^i, q^k)] \times  [\Pilbar(q^k, q^j)]  = 0,
 \end{equation*}
where the contributions to the signs in the second term are given by 
\begin{align*}
 |q^i| \left( |q^k| +1 \right)  + |q^k| \left( |q^j| +1 \right) &   \textrm{ from  the definition of the twisted complex} \\
 (\dim(\Pilbar(q^k, q^j))+1)(|q^k| -|q^i|) & \textrm{ from Equation \eqref{eq:sign_additive_enlargement}} \\ 
  \dim(\Pilbar(q^i, q^k)   &  \textrm{ from Equation \eqref{eq:signs_a_infty_cubical}.}
\end{align*}
Ignoring signs, we see that Equation \eqref{eq:boundary_strips_breaking} implies the desired result.  To verify that the signed formula is correct, the reader should check that the sum of these contributions with $|q^i| \left( |q^j| +1 \right) $  is equal to $1 + |q^i| + |q^k| $ using the formula for the dimension of the moduli spaces given in Proposition \ref{prop:moduli_spaces_strips_manifold}.
\end{proof}
\section{Review of the wrapped Fukaya category} \label{sec:revi-wrapp-fukaya}
\subsection{Preliminaries for Floer theory} \label{sec:review-wrapped}
In \cite{abouzaid-seidel}, we defined the wrapped Fukaya category of a Liouville domain $M$.  In this paper, we use a variant which does not use direct limits, following the construction introduced in \cite{generate} in which we use a  Hamiltonian  function growing quadratically at infinity to define Floer cohomology groups.

Recall that $M$ is a manifold equipped with a Liouville $1$-form $\lambda$.   The Liouville vector field  $Z_{\lambda}$ defined by the equation
\begin{equation*}  i_{Z_{\lambda}} \omega = \lambda \end{equation*}
agrees with the radial vector field $-r \partial_{r}$ along the cylindrical end of $M$, and we write  $\psi^{\rho}$ for the image of the negative Liouville flow for time $\log(\rho)$.  Note that this flow may be written explicitly on the cylindrical end: 
\begin{equation*}
  \psi^{\rho}(r,m) = (\rho \cdot r , m).
\end{equation*}

Let $\sH(M) \subset C^{\infty}(M, \bR) $ denote the set of smooth functions $H$ satisfying
\begin{equation}
  \label{eq:quadratic_growth}
  H(r,y) = r^2
\end{equation}
away from some compact subset of $M$, and write $\sH_{Q}(M)$ for those which in addition vanish on $Q$.  We shall fix such a function for the purposes of defining Floer cohomology, and let $X$ denote the Hamiltonian flow of $H$ defined by the equation 
\begin{equation*}  i_{X} \omega = dH. \end{equation*}

For each pair $L_0, L_1 \in \Ob(\Wrap_{b}(M))$, we define  $\Chord(L_{0},L_1)$ to be the set of time-$1$ flow lines of $X$ which start on $L_0$ and end on $L_1 $, i.e. an element of $x$ is a map  $x \co [0,1] \to M $ such that
\begin{equation*}
  \begin{cases}
  x(0) & \in L_0  \\
x(1) & \in L_1  \\
dx/dt & = X.  
  \end{cases}
\end{equation*}
We shall assume that
\begin{equation}
  \label{eq:non-degenerate_chord}
    \parbox{36em}{all  time-$1$ Hamiltonian chords of $H$ with boundaries on $L_0$ and $L_1$  are non-degenerate.}
\end{equation}
It is  convenient to remember that the elements of $\Chord(L_{0},L_1)$  are in bijective correspondence with intersection points between $L_1$ and the image of $L_0$ under the time-$1$ Hamiltonian flow of $H$.  Moreover, those chords which lie in the complement of $\partial M^{in}$ are in bijective correspondence with Reeb chords with endpoints on the $\partial L_{0}^{in} $ and  $\partial L_{1}^{in} $ which, by Condition \eqref{eq:boundary_Legendrian}, are Legendrian submanifolds of $\partial M^{in}$.

In particular, non-degeneracy of chords lying in $M^{in}$ corresponds to transversality between $L_1$ and the image of $L_0$ under the time-$1$ Hamiltonian flow of $H$,  while non-degeneracy in the cylindrical end corresponds to non-degeneracy of all Reeb chords.  Condition \eqref{eq:non-degenerate_chord} therefore holds after generic Hamiltonian perturbations of the Lagrangians in $\Ob(\Wrap_{b}(M))$, and one may moreover assume that the perturbation preserves Condition \eqref{eq:boundary_Legendrian}.  We shall therefore replace any collection $\Ob(\Wrap_{b}(M))  $ by Hamiltonian isotopic ones which satisfy Condition \eqref{eq:non-degenerate_chord}.  As Lagrangian Floer cohomology is invariant under Hamiltonian perturbations, this results in no loss of generality.

To each element $ x \in  \Chord(L_{0},L_1)$, we assign a Maslov index $|x| $ coming from the gradings on $L_0$ and $L_1$.  This grading agrees with the Maslov index of the intersection of $L_1$ with the image of $L_0$ under the time-$1$ Hamiltonian flow of $H$ which is defined, for example, in Section (11h) of \cite{seidel-book}.  

We define the action of an $X$-chord starting on $L_i$ and ending on $L_j$ by the familiar formula (note that our conventions on action differ by a sign from those of \cite{AS})
\begin{equation} \Action(x) = -\int_{0}^{1} x^{*}(\lambda) + \int H(x(t)) dt + f_{L_j}(x(1)) - f_{L_i}(x(0))  \end{equation}
Equation \eqref{eq:quadratic_growth} implies that any chord that intersects some slice $\partial M^{in} \times \{r\} $ is contained therein, and has action
\begin{equation}  \label{eq:action_formula_cylinder} \Action(x) =  -r^{2} + \ell_{j} - \ell_{i}\end{equation}
where $\ell_{j}$ and $\ell_{i}$ are the values of $f_i$ and $f_j$ on the ends of $L_i$ and $L_j$.
\begin{lem} \label{lem:action_proper}
$\Action$ is a proper map from $ \Chord(L_{0},L_1) $ to $\bR$.
\end{lem}
\begin{proof}
Non-degeneracy implies that there are only finitely many chords in any compact subset of $M$, while Equation \eqref{eq:action_formula_cylinder} implies that the action of a sequence of chords which escapes every compact set must go to $-\infty$. 
\end{proof}

\subsection{Moduli spaces of strips}
Let us fix once and for all  a smooth map
\begin{equation*} \tau \co [0,1] \to [0,1]  \end{equation*}
such that $\tau$ is identically $0$ (respectively $1$) in a neighbourhood of the appropriate endpoint.   Given a pair $x_0, x_1 \in    \Chord(L_{0},L_1)$ we define  $\Disct(x_0;x_1)$ to be the moduli space of maps
\begin{equation*} u \co  Z \to M  \end{equation*}
satisfying the following boundary and asymptotic conditions
\begin{equation*} 
\left\{
\begin{aligned}
 & u \left(\bR \times \{ 1 \} \right) \subset  L_1, \\
 & u \left(  \bR \times \{ 0 \}   \right) \subset L_0, \\
 &  \lim_{s \rightarrow - \infty} u(s, \cdot) = x_{0}( \cdot)  , \\ 
&  \lim_{s \rightarrow + \infty} u(s,\cdot) = x_1(\cdot)
\end{aligned}
\right.
\end{equation*}
as solving Floer's equation
\begin{equation*}  \partial_{s}u = - I_{t} \left(\partial_tu - X \frac{d \tau}{dt} \right). \end{equation*}
We shall write this equation in a coordinate free way as
\begin{equation}
  \label{eq:dbar_strip}
  \left(du - X \otimes d\tau  \right)^{0,1} = 0.
\end{equation}
Since Equation \eqref{eq:dbar_strip} is invariant under translation in the $s$-variable, the reals act on $\Disct(x_0;x_1)$.  The analogue of Theorem \ref{prop:moduli_spaces_strips_manifold} holds as well
\begin{lem} \label{lem:free_R_action_strips}
The moduli space  $\Disct(x_0;x_1)$ is regular for a generic choice of almost complex structures $I_t$ and has dimension $|x_0| - |x_1|$. \noproof
\end{lem}
We write $\Disc(x_0;x_1)$ for the quotient of $\Disct(x_0;x_1)$ by the $\bR$ action whenever it is free, and declare it to be the empty set otherwise.

By adding \emph{broken strips} to this moduli space, we obtain a manifold with boundary   $\Discbar(x_0; x_1)$ whose strata are disjoint unions over all sequences starting with $x_0$ and ending with $x_1$ of the products of the moduli spaces
\begin{equation}
 \Disc(x_0; y_1) \times   \Disc(y_1; y_2)  \times \cdots \times  \Disc(y_k; y_{k+1})  \times  \Disc(y_{k+1}; x_1) 
\end{equation}
Since the Lagrangians $L_0$ and $L_1$ have infinite ends, it does not immediately follow from Gromov compactness that $\Discbar(x_0; x_1)  $ is compact. 

\begin{lem} \label{lem:gromov_compactness}
If a solution to Equation \eqref{eq:dbar_strip} converges at $-\infty$ to $x_0$ and at $+\infty$ to $x_1$, then
\begin{equation} \Action(x_0) \geq \Action(x_1). \end{equation}
Moreover, all such solutions lie entirely within a compact subset of $M$ depending only on $x_0$ and $x_1$.
\end{lem}
\begin{proof}
The first part is a standard energy estimate using the positivity of $H$.  To prove the second, we appeal to the argument given in Lemma 7.2 of \cite{abouzaid-seidel}:  consider any hypersurface $\partial M^{in} \times \{ r \}$ separating $x_0$ and $x_1$ from infinity.  If a solution to Equation \eqref{eq:dbar_strip} escaped this region, we would find a compact surface $\Sigma$, mapping to the cone $\partial M^{in} \times [r , +\infty)$ and with boundary conditions the concave end $\{ r \} \times \partial M^{in}$ and a collection of Lagrangians on which $\lambda$ vanishes identically.  Applying Stokes' theorem, and using the fact that $H | \partial M^{in} \times [r , +\infty)$ achieves its minimum on the boundary, we find that
\begin{equation} 0 <  \int_{\partial_{r} \Sigma} \lambda \circ( du -  X \otimes (wd\tau + \beta)) =  \int_{\partial_{r} \Sigma} (\lambda \circ J_{z}) \circ ( du -  X \otimes (wd\tau + \beta)) \circ j.   \end{equation}
where $\partial_{r} \Sigma$ is the inverse image of  $\partial M^{in} \times \{ r \}$.  Since the restriction of $X$ to  $\{ r \} \times \partial M^{in}$ is a multiple of the Reeb flow, $\lambda(J_{z} X)$ vanishes, so we conclude that
\begin{equation} 0 <  \int_{\partial_{r} \Sigma} \lambda \circ J_{z} \circ du  \circ j.   \end{equation}
On the other hand, if $\xi$ is a tangent vector compatible with the natural orientation of $\partial_{r} \Sigma$, $j \xi$ is inward pointing along $\partial S$, so that $du( j \xi)$ points towards the infinite part of the cone.  The condition that $J_{z}$ be of contact type implies that
\begin{equation} \lambda  ( J_{z}\circ  du \circ j \xi) \leq 0, \end{equation}
yielding a contradiction.
\end{proof}
Together with Lemma \ref{lem:action_proper}, this result implies that for each chord $x_1$, there is a compact subset of $M$ containing the union of all the moduli spaces $\Disc(\cdot;x_1)$.  Applying the usual Gromov compactness and gluing theory results to this moduli space, we conclude:
\begin{cor}
 \label{cor:compactification_strip_manifold}
For each chord $x_1$, the moduli space $\Discbar(x_0;x_1)$ is empty for all but finitely many choices of $x_0$, and is a compact manifold with boundary of dimension $|x_0|- |x_1| -1$ whenever $I_t$ is a generic family of almost complex structure.  Moreover, the boundary is covered by the closure of the images of the natural inclusions
\begin{equation*}   \Disc(x_0; y) \times   \Disc(y; x_1) \to     \Discbar(x_0; x_1)  \end{equation*} 
\end{cor}

From now on, we shall fix an almost complex structure $I_t$ for which the conclusion of this Corollary holds.
\subsection{The wrapped Floer complex}
We define the graded vector space underlying the Floer complex to be the direct sum
\begin{equation} \label{eq:floer_complex} CW^{*}_{b}(L_0,L_1) = \bigoplus_{x \in \Chord(L_0,L_1)} |\ro_{x}| .\end{equation}
Here, $\ro_x$ is a certain real vector space of rank $1$ associated to every chord via a construction briefly reviewed in Appendix \ref{sec:orientations}, and  $|\ro_{x}|$ is generated by the two possible orientations on it.  Unless the reader wants to check the signs, there is no harm in assuming that we are simply taking the vector space freely generated by the set of chords.

The differential is then a count of solutions to Equation \eqref{eq:dbar_strip}.  More precisely, whenever $|x_1| = |x_0| +1$, every element $u$ of the moduli space $\Disc(x_0,x_1)$ is rigid, and defines an isomorphism 
\begin{equation*}
  \ro_{x_1} \to \ro_{x_0}
\end{equation*}
as defined in  Section \ref{sec:signs}. In particular, we may induce an orientation of $\ro_{x_0}$ from one on $   \ro_{x_1} $, and the associated map on orientation lines is denoted $\mu_u$.  We define
\begin{align}
  \label{eq:wrapped_differential}
 \mu^{\F}_1 \co  CW^{i}_{b}(L_0,L_1)  & \to CW^{i+1}_{b}(L_0,L_1)  \\
[x_1] & \mapsto (-1)^i \sum_{u} \mu_u([x_1] ).
\end{align}
Ignoring signs, we are indeed simply counting elements of the moduli spaces $\Disc(x_0;x_1) $.   The proof that this count gives a well-defined differential is standard.  The only possibly new phenomenon is that Corollary \ref{cor:compactification_strip_manifold} implies that each chord can be the input of only finitely many solutions to \eqref{eq:dbar_strip}, which guarantees that the image of the corresponding generator is a sum of only finitely many terms.

Let us note that the graded vector space underlying $  CW^{*}_{b}(L_0,L_1)  $ depends on $L_0$, $L_1$, $H$ and $\omega$, while the differential also depends on $I_{t}$.  We write
\begin{equation*}
   CW^{*}_{b}(L_0,L_1; \omega, H, I_{t}) 
\end{equation*}
when the distinction is important as shall be the case in the next result. 
\begin{lem}
If $\psi \co M \to M$ satisfies $\psi^{*}(\omega) = \rho \omega$ for some non-zero constant $\rho$, then we have a canonical isomorphism
\begin{equation} \label{eq:isomorphic_complexes_rescale}
   CW(\psi) \co   CW^{*}_{b}(L_0,L_1; \omega, H, I_{t})  \cong   CW^{*}_{b}\left(\psi (L_0), \psi( L_1); \omega, \frac{H}{\rho}\circ  \psi , \psi^{*} I_{t}\right) 
\end{equation}
\end{lem}
\begin{proof}
For any diffeomorphism $\psi$, we obtain an isomorphism of chain complexes
\begin{equation*}
       CW^{*}_{b}(L_0,L_1; \omega, H, I_{t})  \cong  CW^{*}_{b}\left(\psi (L_0), \psi( L_1); \psi^{*} \omega, H \circ \psi , \psi^{*} I_{t}\right) 
\end{equation*}
by composing every chord from $L_0$ and $L_1$ with $\psi$ to obtain a chord from $ \psi (L_0)$ to $ \psi (L_1) $, and every solution to Equation \eqref{eq:dbar_strip} with $\psi$ to obtain an analogous solution for the family of almost complex structure $  \psi^{*} I_{t} $.  The property that $\psi$ rescale $\omega$ implies that the Hamiltonian flow of $H \circ \psi$ with respect to $  \psi^{*} \omega $ agrees with the flow of $ \frac{H}{\rho}\circ  \psi  $ with respect to $\omega$.  Since only the Hamiltonian flow appears in Equation \eqref{eq:dbar_strip}, we obtain an identification
\begin{equation*}
 CW^{*}_{b}(\psi( L_0), \psi (L_1); \psi^{*} \omega, H \circ \psi , \psi^{*} I_{t})  \cong  CW^{*}_{b}\left(\psi( L_0), \psi( L_1) ; \omega, \frac{H}{\rho}\circ  \psi , \psi^{*} I_{t}\right) 
\end{equation*}
which proves the desired result.
\end{proof}
From now on, we define
\begin{equation*}
    CW^{*}_{b}(\psi (L_0), \psi( L_1)) \equiv CW^{*}_{b}\left(\psi (L_0), \psi (L_1) ; \omega, \frac{H}{\rho}\circ  \psi , \psi^{*} I_{t}\right) .
\end{equation*}

\subsection{Composition in the wrapped Fukaya category} \label{sec:comp-wrapp-fukaya}
Given a triple of Lagrangians $L_0$, $L_1$ and $L_2$, we  shall define a map
\begin{equation} \label{eq:pre-multiplication}
\mu^{\psi^{2}}_{2} \co CW^{*}_{b}(L_1, L_2)  \otimes CW^{*}_{b}(L_0,L_1)  \to   CW^{*}_{b}\left( \psi^{2} (L_0), \psi^{2}( L_2) \right),
\end{equation}
where $\psi^{2}$ is the time $\log(2)$ negative Liouville flow. After composition with the inverse of the isomorphism of Equation \eqref{eq:isomorphic_complexes_rescale}, we obtain the product
\begin{equation}
 \label{eq:cup_product}
\mu^{\F}_2 \co CW^{*}_{b}(L_1, L_2)  \otimes CW^{*}_{b}(L_0,L_1)  \to   CW^{*}_{b}( L_0,  L_2)
\end{equation}
in the wrapped Fukaya category.
\begin{figure}
  \centering
   \includegraphics{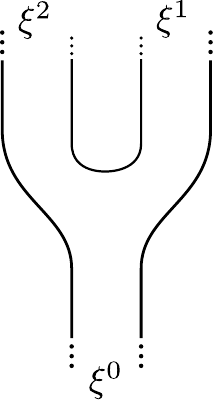}
  \caption{}
  \label{fig:disc_2_inputs}
\end{figure}
The map \eqref{eq:pre-multiplication} shall count solutions to a Cauchy-Riemann equation
\begin{equation}
  \label{eq:dbar_pair_pants}
  \left( du - X_S \otimes \alpha_S \right)^{0,1} = 0.
\end{equation}
whose source is the surface $S$ obtained by removing $3$ points  $(\xi^0,\xi^1,\xi^2)$ from the boundary of $D^2$ (see Figure \ref{fig:disc_2_inputs}).  In the above equation, $\alpha_S$ is a closed $1$ form on $S$ while $X_S$ is the Hamiltonian vector field of a function $H_{S}$ on $M$ which depends on $S$.    The most important condition on these data is the requirement that $H_{S}$ come from a map
 \begin{equation*}
    H_{S} \co S \to \sH(M) 
  \end{equation*}
and restrict to $H$ near  $\xi^1$ and $\xi^{2}$  and to  $\frac{H}{4} \circ \psi^{2}$ near $\xi^0$.  To see that this makes sense, we note that $\frac{H}{4} \circ \psi^{2}$ indeed lies in $\sH(M)$ since $H$ agrees with $r^2$ on the cylindrical end.

To specify the remaining data required for Equation \eqref{eq:dbar_pair_pants},  we choose strip-like ends for $S$, i.e. we write  $Z_+$ and $Z_-$ for the positive and negative half-strips  in $Z$ and choose embeddings 
\begin{align*}
\epsilon^{0} \co Z_- & \to S \\
 \epsilon^k \co Z_+ & \to S \textrm{ if k=1,2} 
\end{align*}
which map $\partial Z_{\pm}$ to $\partial S$ and converge to the respective marked points $\xi^k$.

The closed $1$-form $\alpha_S$ is required to vanish on $\partial S$, and to satisfy
\begin{align*}
{\epsilon^{0}}^{*}(\alpha_S) & = 2 d \tau \\
{\epsilon^k}^{*} (\alpha_S) & =  d \tau \textrm{ if k=1,2}.
\end{align*}
In addition, we choose a family of almost complex structures
\begin{align*}
  I_{S} & \co S  \to \sJ(M)
\end{align*}
whose compositions with $\epsilon^k$ agrees with $I_{t}$ if $k = 1,2$, and with $(\psi^{2})^{*}I_{t}$ if $k=0$.

We would like the Lagrangian boundary conditions to be given by $(L_0, L_1)$ near $\xi^{1}$, $(L_1, L_2)$ near $\xi^{2}$,  and $(\psi^{2}(L_0), \psi^{2}( L_2) )$ near $\xi^0$.  Along the two segments of $ \partial S$ converging to $\xi^0$, we cannot therefore have a constant Lagrangian condition, but we must interpolate between a Lagrangian and its image under $\psi^2$.  Technically, this puts us in the framework of \emph{moving Lagrangian boundary conditions} (see Section (8k) of \cite{seidel-book}).    We shall choose the simplest such moving boundary condition by fixing  a map $\rho_{S}$ from the boundary of $D^2$ to the interval  $[1,2]$   such that
\begin{equation}
  \label{eq:condition_arc_labelling}
  \parbox{36em}{$\rho_S(z) \equiv 1$ if $z$ is near $\xi^{1}$ or  $\xi^{2}$ and $\rho_S(z) \equiv 2$ if $z$ is near $\xi^0$.} 
\end{equation}

\begin{defin}
 The moduli space $\Disc_{2}( x_0;x_1,x_2)$ is the space of solutions to Equation \eqref{eq:dbar_pair_pants} with boundary conditions
\begin{equation}
  \label{eq:boundary_disc_2_punctures}
  \begin{cases}  u(z)  \in \psi^{\rho_{S}(z)} (L_1) & \textrm{if $z \in \partial S$ lies between $\xi_1$ and $\xi_2$} \\
  u(z)  \in \psi^{\rho_{S}(z)} (L_2) & \textrm{if $z \in \partial S$ lies between $\xi_2$ and $\xi_0$} \\
  u(z)  \in \psi^{\rho_{S}(z)} (L_0) & \textrm{if $z \in \partial S$ lies between $\xi_1$ and $\xi_0$}
 \end{cases}
\end{equation}
and such that  the image of $u$ converges to $x_{1}$ and $x_2$ at the corresponding incoming strip-like ends, and to $  \psi^{2}(x_0)$  at the outgoing strip-like end.
\end{defin}

By construction, we have ensured that the pullback of Equation \eqref{eq:dbar_pair_pants} by the strip-like end $\xi^k$ is given by
\begin{align*}
  (du - X \otimes d\tau)^{0,1} = 0 & \textrm{ with respect to $I_t$ if $k=1,2$} \\
   (du - 2 X_{\frac{H}{4} \circ \psi^{2}} \otimes d\tau)^{0,1} = 0 & \textrm{ with respect to $(\psi^{2})^{*}I_{t}$ if $k=0$}
\end{align*}
with Lagrangian boundary conditions  $(L_0, L_1)$ near $\xi^{1}$, $(L_1, L_2)$ near $\xi^{2}$,  and $(\psi^{2}(L_0) , \psi^{2}( L_2) )$ near $\xi^0$.  For the inputs, we exactly recover Equation \eqref{eq:dbar_strip}, while for the output, we recover the same Equation, but for the Hamiltonian $ \frac{H}{2} \circ \psi^{2} $.  We conclude that the Gromov bordification of the moduli space $\Disc_{2}( x_0;x_1,x_2)$  is obtained by adding the strata
\begin{align}
 \label{eq:cover_boundary_disc_2}
& \coprod_{y \in \Chord(L_0,L_1)}  \Disc_{2}( x_0;y,x_2) \times \Discbar(y;x_1)   \\
& \coprod_{y \in \Chord(L_1,L_2)}  \Disc_{2}( x_0;x_1,y) \times \Discbar(y;x_2)   \\  \label{eq:cover_boundary_disc_2-end}
& \coprod_{y \in \Chord(L_0,L_2)}  \Discbar( x_0;  y) \times \Disc_{2}( y;x_1,x_2)    \end{align}
corresponding to the breaking of a holomorphic strip at any of the ends. The factor  $ \Discbar( x_0;  y) $ in the third stratum is obtained by applying the inverse of $\psi^{2}$ to the moduli space $  \Discbar(\psi^{2} (x_0) ; \psi^{2} (y)) $  which would appear naturally from the point of view of breakings of holomorphic curves. 
\begin{lem} \label{lem:regularity_compactness_2_inputs}
For a fixed pair $(x_1,x_2)$, the moduli space $ \Disc_{2}( x_0;x_1,x_2) $  is empty for all but finitely many choices of chords $x_0$.  For a generic family of almost complex structures $I_S$ and Hamiltonians $H_{S}$,  $\Discbar_{2}( x_0;x_1,x_2)$ is a compact manifold of dimension
\begin{equation*}
  |x_0| - |x_1| - |x_2|
\end{equation*}
whose boundary is covered by the codimension $1$ strata listed in Equations  \eqref{eq:cover_boundary_disc_2}-\eqref{eq:cover_boundary_disc_2-end}. \noproof
\end{lem}
The proof of transversality follows from a standard Sard-Smale argument going back to \cite{FHS} in the case of Hamiltonian Floer cohomology.  The proof of compactness relies an argument analogous to that of Lemma \ref{lem:gromov_compactness}, and which is explained in the proof of Lemma 3.2 of \cite{generate}.

Whenever $  |x_0|  = |x_1| + |x_2|$,  $\Disc_{2}(x_0;x_1,x_2)$ consists of finitely many elements, each of which defines an isomorphism
\begin{equation*}
 \ro_{x_2} \otimes \ro_{x_1} \to \ro_{  \psi^{2} (x_0)},
\end{equation*}
coming from  Equation \eqref{eq:orientation_higher_product}.  Writing $\mu_u$  as before for the map induced on orientation lines we define a  product on the wrapped Floer complex by adding the appropriate signed contribution of each element of  $ \Disc_{2}(x_0;x_1,x_2) $:
 \begin{align*}
\mu^{\psi^2}_{2} \co CW^{*}_{b}(L_1, L_2)  \otimes CW^{*}_{b}(L_0,L_1 ) & \to   CW^{*}_{b}( \psi^{2} (L_0) , \psi^{2} (L_2))  \\
\mu^{\psi^2}_2([x_2],[x_1]) & = \sum_{ \substack{ |x_0|  = |x_1| + |x_2| \\  u \in \Disc_{2}( x_0;x_1,x_2)} } (-1)^{|x_1|}  \mu_u([x_2],[x_1]).
\end{align*}

\subsection{Failure of associativity}
Let us assign  to each end $\xi^i$ of a punctured disc a \emph{weight} $w_i$ which is a positive real number, with the convention that any $\dbar$ operator we shall consider on such a disc must pull back, under a strip-like end near $\xi^i$, to Equation \eqref{eq:dbar_strip} up to applying $\psi^{w_i}$.  In the case of a disc with $2$ incoming ends, the inputs have weights $1$ while the output has weight $2$.    If we consider the pull back by $\psi^2$ of all the data used to define the moduli spaces $\Disc_{2}( x_0;x_1,x_2)  $ in the previous section, we obtain a $\dbar$ operator such that the weights are now equal to $2$ at the inputs, and $4$ at the output.   The space of solutions of the associated Cauchy-Riemann equation is a moduli space $\Disc_{2}(\psi^{2} x_0; \psi^{2} x_1, \psi^{2} x_2)$ the count of whose elements defines a map
\begin{equation} \label{eq:product_rescaled_4}
\mu^{\psi^4}_{2} \co CW^{*}_{b}(\psi^{2} (L_1),  \psi^{2} (L_2)) \otimes CW^{*}_{b}( \psi^{2}( L_0), \psi^{2}( L_1 ))  \to   CW^{*}_{b}( \psi^{4}( L_0), \psi^{4} (L_2))   
\end{equation}
which exactly agrees with $\mu^{\F}_2$ if we pre-compose with $ CW(\psi^{2}) $ on both factors the source and post-compose with $  CW(\psi^{1/4}) $.

In the next few sections, we shall prove that $\mu^{\F}_2$ induces an associative product on cohomology by constructing a homotopy between the two possible compositions around the diagram
\tiny
\begin{equation} \label{eq:htpy_associativity_diagram}
  \xymatrix{  CW^{*}_{b}(L_2, L_3) \otimes CW^{*}_{b}(L_1, L_2)  \otimes CW^{*}_{b}(L_0,L_1) \ar[d]^{CW(\psi^{2}) \otimes\mu^{\psi^2}_{2} } \ar[r]^{ \mu^{\psi^2}_{2}   \otimes CW(\psi^{2})}   &   CW^{*}_{b}(\psi^{2} (L_1), \psi^{2} (L_3)) \otimes  CW^{*}_{b}(\psi^{2} ( L_0) ,\psi^{2} ( L_1))  \ar[d]^{ \mu^{\psi^4}_{2} }   \\
 CW^{*}_{b}(\psi^{2} (L_2), \psi^{2}( L_3)) \otimes  CW^{*}_{b}(\psi^{2} ( L_0) ,\psi^{2}  (L_2))   \ar[r]^{  \mu^{\psi^4}_{2} } &  CW^{*}_{b}(\psi^{4}( L_0), \psi^{4} (L_3)) \ar[d]^{\cong}. \\
 & CW^{*}_{b}(L_0, L_3) }
\end{equation}

\normalsize
\comment{
\begin{equation} \label{eq:htpy_associativity_diagram}
  \xymatrix{   CW^{*}_{b}(L_2, L_3) \otimes  CW^{*}_{b}( L_1,  L_2) \otimes CW^{*}_{b}(\psi^{2} ( L_0),\psi^{2} ( L_1)) \ar[r]^{ \mu^{\psi^2}_{2} \otimes \id } & CW^{*}_{b}(\psi^{2} (L_1), \psi^{2} (L_3)) \otimes  CW^{*}_{b}(\psi^{2} ( L_0) ,\psi^{2} ( L_1))  \ar[d]^{ \mu^{\psi^2}_{4}} \\
CW^{*}_{b}(L_2, L_3) \otimes CW^{*}_{b}(L_1, L_2)  \otimes CW^{*}_{b}(L_0,L_1) \ar[d]^{CW(\psi^{2}) \otimes \id^{\otimes 2} }  \ar[u]^{ \id^{\otimes 2} \otimes CW(\psi^{2})}   &  CW^{*}_{b}(\psi^{4}( L_0), \psi^{4} (L_3)).  \\
CW^{*}_{b} (\psi^{2} (L_2), \psi^{2} (L_3)) \otimes  CW^{*}_{b}( L_1,  L_2) \otimes CW^{*}_{b}(  L_0,  L_1) \ar[r]^{\id \otimes  \mu^{\psi^2}_{2} } &  CW^{*}_{b}(\psi^{2} (L_2), \psi^{2}( L_3)) \otimes  CW^{*}_{b}(\psi^{2} ( L_0) ,\psi^{2}  (L_2))   \ar[u]^{  \mu^{\psi^2}_{4} } . } 
\end{equation}}

It is well known that the product $\mu^{\F}_{2}$ is not in general associative  at the chain level, and that such a homotopy should come from a moduli space of maps whose sources are $4$-punctured discs equipped with an arbitrary conformal structure.   We write $\Disc_{3}$ for this moduli space, and recall that $\Discbar_{3}$  is an interval whose two endpoints are nodal discs obtained by gluing, in the two possible different ways as represented in the outermost surfaces in Figure \ref{fig:Three_puncture}, two discs each with two incoming ends and one outgoing one.

\begin{figure}
  \centering
    \includegraphics{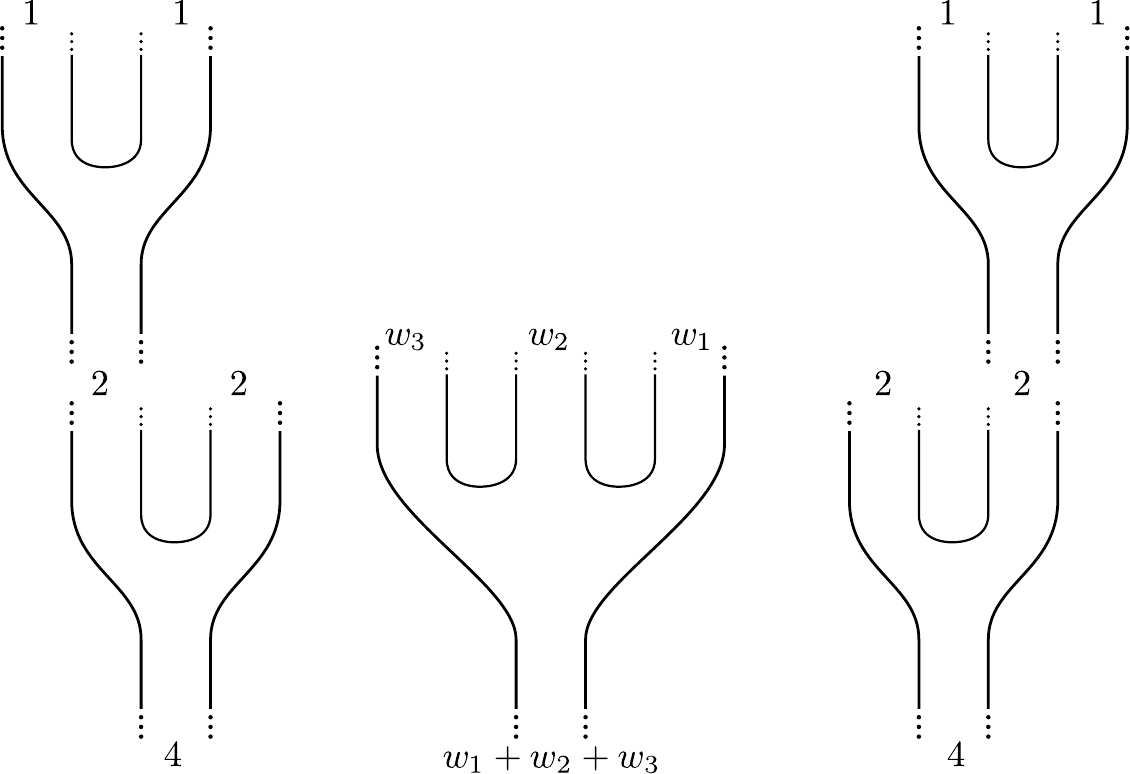}
  \caption{ }
  \label{fig:Three_puncture}
\end{figure}

In order for this count to indeed define a homotopy, whatever equation we define on the moduli space $\Discbar_{3}$ is usually required to restrict on the boundary strata to Equation \eqref{eq:dbar_pair_pants} on each component.  In particular, the Cauchy-Riemann equation imposed on a disc whose conformal equivalence class is close to the boundary of $\Discbar_{3}$ should be obtained by gluing the $\dbar$ operators associated to Equation \eqref{eq:dbar_pair_pants} at the node.  This only makes sense if the restriction of these $\dbar$ operators  to the two strip-like ends at the node agree, which is not the case for us, as even the Lagrangian boundary conditions do not agree.  However, they agree up to applying $  \psi^{2} $, so we may glue two solutions to Equation  \eqref{eq:dbar_pair_pants} after applying $  \psi^{2} $ to one of them.    This is the main reason for stating associativity in term of the Diagram \eqref{eq:htpy_associativity_diagram}.

As a final observation, we note that the weights on the inputs of a $3$-punctured disc must be allowed to vary with its modulus, for they are given by $(1,1,2)$ at one of the endpoints of $\Discbar_{3}$ and $(2,1,1)$ at the other.   We could now choose auxiliary data of families  of Hamiltonians and almost complex structures on $M$, and of $1$-forms on elements of  $\Discbar_{3}$ which would define an operation
\begin{equation*}
    CW^{*}_{b}(L_2, L_3) \otimes CW^{*}_{b}(L_1, L_2)  \otimes CW^{*}_{b}(L_0,L_1)  \to  CW^{*}_{b}( L_0, L_3)
\end{equation*}
providing the homotopy in Diagram \eqref{eq:htpy_associativity_diagram}.  As we shall have to repeat this procedure for an arbitrary number of inputs, and there is no simplifying feature in having only three, we proceed to give the general construction.
\subsection{Floer data for the $A_{\infty}$ structure} 
We write $ \Disc_{d} $ for the moduli space of abstract discs with one negative puncture denoted $\xi^{0}$ and  $d$ positive punctures  denoted $\{ \xi^{k} \}_{k=1}^{d}$ which are ordered clockwise.  We write $  \Discbar_{d} $  for the Deligne-Mumford compactification of this moduli space, and assume that a \emph{universal and consistent} choice of strip-like end has been chosen as in Section (9g) of \cite{seidel-book}.  This means that we have, for each surface $S$ and each puncture, a map
\begin{equation*} \epsilon^k \co Z_{\pm}  \to S \end{equation*}
whose source is $Z_-$ if $k=0$ and $Z_+$ otherwise, and that such a choice varies smoothly with the modulus of the surface $S$ in the interior of the moduli space.  Moreover,  near each boundary stratum $\sigma$ of $  \Discbar_{d} $ the strip-like ends are obtained by gluing in the following sense:  if we write $\sigma = \Disc_{d_1} \times \cdots \times  \Disc_{d_j}$, then gluing the strip-like ends chosen on these lower dimensional moduli spaces defines an embedding
\begin{equation*}\label{eq:corner_chart_moduli_space}
  \sigma \times [1,+\infty)^{j-1} \to   \Disc_{d}.
\end{equation*}
A surface in the image of this gluing map is by definition covered by patches coming from surfaces in one of the factors of the product decomposition of $\sigma$.  In particular, each end of such a surface obtained by gluing comes equipped with strip-like ends induced from the choices of strip-like ends on the lower dimensional moduli spaces.   Consistency of the choice of strip-like ends is the requirement that our choice on $  \Disc_{d} $  agree with this fixed choice in some neighbourhood of $\sigma$.
\begin{defin} \label{def:floer_datum_disc_1_output}
A \emph{Floer datum} $D_{S}$ on a stable disc $S \in \Discbar_{d}$ consists of the following choices:
\begin{enumerate}
\item Weights: A positive integer $w_{k,S}$ assigned to the $k$\th end such that
 \begin{equation*} w_{0,S} = \sum_{1 \leq k \leq d} w_{k,S} \end{equation*}
\item Moving conditions:  A map $ \rho_{S} \co \partial \bar{S} \to [1,+ \infty)$ which agrees with $w_k$ near the $k$\th end.
\item Basic $1$-form: A closed $1$-form $\alpha_{S}$ whose restriction to the boundary vanishes and whose pullback under  $ \epsilon^k$ agrees with $ w_{k,S} d\tau$.
\item  Hamiltonian perturbations:   A map $H_{S} \co S \to \sH(M)$ which agrees with  $ \frac{H \circ \psi^{w_{k,S}}}{w^{2}_{k,S}}$  near the $k$\th end.
\item Almost complex structure:  A map $I_{S} \co S \to \sJ(M)$ whose pullback under $ \epsilon^k $ agrees with  $ (\psi^{w_{k,S}})^{*} I _t$.
\end{enumerate}
\end{defin}
 
If we write $X_{S}$  for the Hamiltonian flow of $H_{S}$, then these data allow us to write down a Cauchy-Riemann equation
\begin{equation}
  \label{eq:dbar_disc_1_output}
\left(du - X_{S} \otimes \alpha_{S}\right)^{0,1} = 0
\end{equation}
where the $(0,1)$ part is taken with respect to $I_{S}$.

The main reason of the long list of conditions in Definition \ref{def:floer_datum_disc_1_output} is the following conclusion
\begin{lem} \label{lem:equation_restricts_to_ends}
 The pullback of Equation \eqref{eq:dbar_disc_1_output} under $\epsilon^{k}$ is given by
\begin{equation}  
\label{eq:dbar_disc_1_output_restrict}
\left(du \circ \epsilon^{k} - X_{ \frac{H \circ \psi^{w_{k,S}}}{w_{k,S}} } \otimes d \tau \right)^{0,1} = 0.
\end{equation}
In particular, it agrees with Equation \eqref{eq:dbar_strip} up to applying  $ \psi^{w_{k,S}} $.  \noproof
\end{lem}

At the boundary of the moduli space $   \Discbar_{d}$, we should require that Floer data be given by the choices performed on smaller dimensional moduli spaces.  When $d=3$, we already noted in the previous section that the choice of Floer data at the boundary cannot be given exactly by the Floer data for $   \Disc_{2}$ on each component, since these Floer data cannot be glued.  We shall consider the following notion of equivalence among Floer data which is weaker than equality:
\begin{defin}
 We say that a pair $\left( \rho_{S}^{1}, \alpha_{S}^{1}, H_{S}^{1}, I_{S}^{1} \right) $ and  $\left( \rho_{S}^{2}, \alpha_{S}^{2}, H_{S}^{2}, I_{S}^{2} \right) $ of Floer data on a surface $S$  are \emph{conformally equivalent} if there exists a constant $C$ so that $\rho^{2}_{S}$ and $\alpha_{S}^{2}$ respectively agree with   $C \rho^{1}_{S}$ and $C \alpha_{S}^{1}$, and
\begin{align*}
 I_{S}^{2} & = {\psi^{C}}^{*} I_{S}^{1} \\
H_{S}^{2} & = \frac{H_{S}^{1} \circ \psi^{C}}{C^{2} }.
 \end{align*}
\end{defin}
We have already encountered the idea that rescaling by $\psi^{2}$ in Equation \eqref{eq:product_rescaled_4} gives an identification of moduli spaces.  This idea generalises as follows:
\begin{lem}
Composition with $\psi^{C}$  defines a bijective correspondence between solutions to Equation \eqref{eq:dbar_disc_1_output} for conformally equivalent Floer data. \noproof
\end{lem}

We can now state the desired compatibility between Floer data at the boundary of the moduli spaces $\Discbar_{d}$:
\begin{defin}
A   \emph{universal and conformally consistent}  choice of Floer data $\bfD_{\mu}$  for the $A_{\infty}$ structure,  is a choice  of Floer data  for every element of  $\Discbar_{d}$ and every integer $d \geq 2$,  which varies smoothly over the interior of the moduli space, whose restriction to a boundary stratum is conformally equivalent to the product of Floer data coming from lower dimensional moduli spaces, and which near such a boundary stratum agrees to infinite order, in the coordinates \eqref{eq:corner_chart_moduli_space}, with the Floer data obtained by gluing.
\end{defin}
Note that conformal consistency fixes, up to a constant depending on the modulus, the choice of Floer data on  $\partial \Discbar_{d}$ once Floer data on each irreducible component has been chosen.  For example, the choice of Floer data for $d=2$  from Section \ref{sec:comp-wrapp-fukaya} determines the Floer data on the boundary of $ \Discbar_{3}$ as discussed in the previous section, where we chose to use the original Floer data on the disc that does not contain the original end, and rescale it by $\psi^{2}$ on the other disc (see Figure \ref{fig:Three_puncture}).  After introducing a perturbations of this Floer data which vanishes to infinite order at the boundary, we extend it from a neighbourhood of the boundary to the remainder of  $\Disc_{3}$.    One may proceed inductively to construct a universal datum $ \bfD_{\mu} $ using the fact that  the boundary of $\Discbar_{d}$ is covered by the images of codimension $1$ inclusions
\begin{equation}
\Discbar_{d_1} \times \Discbar_{d-d_1+1} \to  \partial \Discbar_{d}
\end{equation}
and the contractibility of the space of Floer data on a given surface.  The contractibility of this space also implies that we may extend any Floer data chosen on a given surface to universal data:
\begin{lem} \label{lem:abundance_floer_data}
The restriction map from the space of universal and conformally consistent Floer data to the space of Floer data for a fixed surface $S$ is surjective.  \noproof
\end{lem}

\subsection{Higher products} \label{sec:a_infty-structure}
In this section, we construct  the $A_{\infty}$ structure on the wrapped Fukaya category which is given by higher products
\begin{equation*}
\mu^{\F}_d \co  CW^{*}_{b}(L_{d-1}, L_d)    \otimes \cdots \otimes    CW^{*}_{b}(L_1, L_2)  \otimes  CW^{*}_{b}(L_0,L_1 )  \to   CW^{*}_{b}( L_0, L_d)
\end{equation*}
coming from the count of certain solutions to Equation \eqref{eq:dbar_disc_1_output}.

Given a sequence of chords $\vx = \{ x_{k} \in \Chord(L_{k-1},L_k) \} $ if $1 \leq k \leq d  $ and $x_0 \in \Chord(L_0,L_d)$ with $L_0, \ldots,  L_{d}$ objects of $\Wrap_{b}(M)$,  we define $\Disc_{d}(x_0; \vx) $  to be the space of solutions to Equation \eqref{eq:dbar_disc_1_output} whose source is an arbitrary element $S \in \Disc_{d}$ with marked points $(\xi^0, \ldots, \xi^{d}) $, such that
\begin{equation}
\lim_{s \to \pm \infty} u \circ \epsilon^{k}(s, \cdot) =  \psi^{w_{k,S}} x_k 
\end{equation}
and with boundary conditions
\begin{equation} \label{eq:boundary_condition_disc_A_infty}
u(z) \in \psi^{\rho_{S}(z)} (L_k) \textrm{ if $z \in \partial S$ lies between $\xi^k$ and $\xi^{k+1}$.} 
\end{equation}
Note that these conditions make sense because of Lemma \ref{lem:equation_restricts_to_ends} which shows that Equation \eqref{eq:dbar_disc_1_output} restricts to Equation \eqref{eq:dbar_strip} up to applying $ \psi^{w_{k,S}}  $.  Figure \ref{fig:Boundary_conditions} shows the asymptotic conditions for $d=3$.
\begin{figure}
  \centering
    \includegraphics{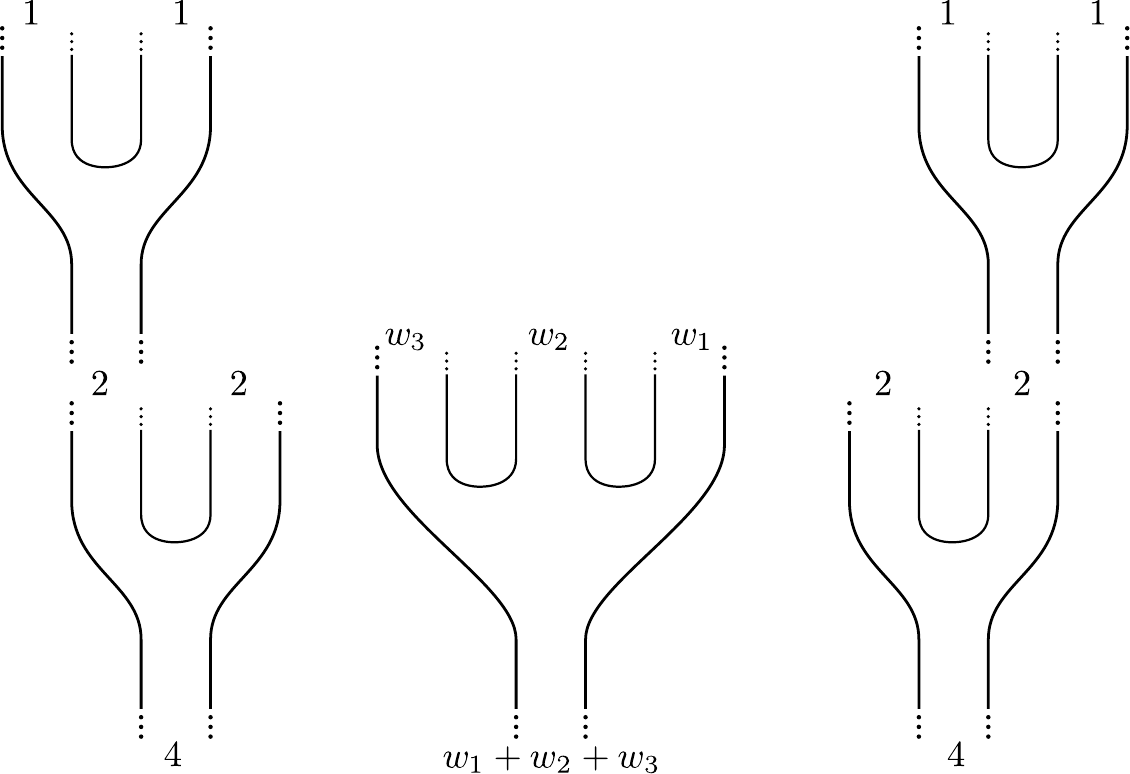}
  \caption{ }
  \label{fig:Boundary_conditions}
\end{figure}
 
With the exception of strips breaking at the ends, the virtual codimension $1$ strata of the Gromov bordification $\Discbar_{d}(x_0; \vx) $ lie over the codimension $1$ strata of $\Discbar_{d}$.  The consistency condition imposed on  $\bfD_{\mu}$ implies that whenever a disc breaks, each component is a solution to Equation \eqref{eq:dbar_pair_pants} for the Floer data $\bfD_{\mu}$  up to applying $\psi^{C}$ for some constant $C$ that depends on the modulus in $\Discbar_{d}$.  Since composition with $\psi^{C}$ identifies the solutions of this rescaled equation with the moduli space of the original equation, we conclude that for each integer $k$ between $0$ and $d-d_2$ and chord $y \in \Chord(L_{k+1},L_{k+d_2}) $, we obtain a natural inclusion
\begin{equation}
  \label{eq:codim_1_strata_discs_1_output}
  \Discbar_{d_1}(x_0; \vx[1])  \times   \Discbar_{d_2}(y; \vx[2]) \to  \Discbar_{d}(x_0; \vx)
\end{equation}
where the sequences of inputs in the respective factors are given by $\vx[2]= (x_{k+1}, \ldots, x_{k+d_2}) $ and $\vx[1] = (x_1, \ldots, x_k, y, x_{k+d_2+1}, \ldots, x_{d})$ as in Figure \ref{fig:breaking}.
\begin{figure}
  \centering
      \includegraphics{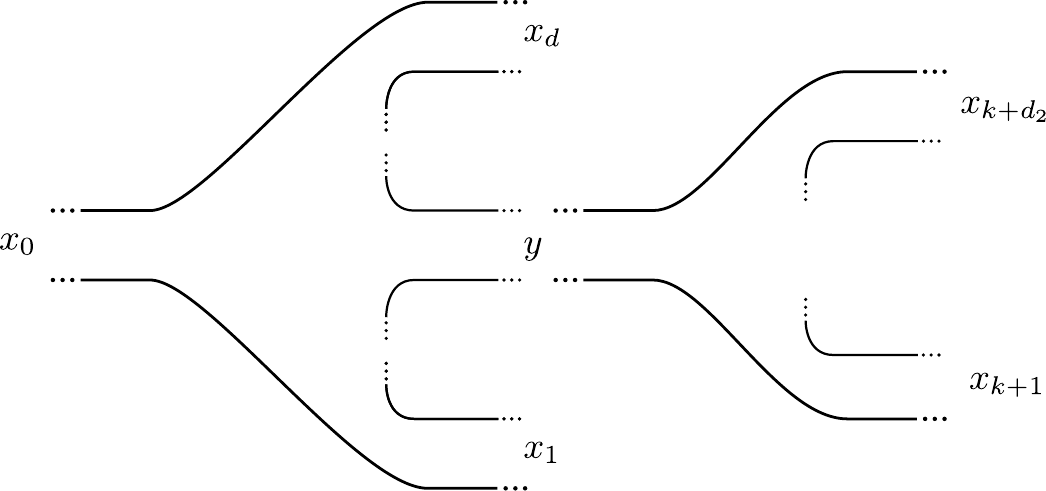}
  \caption{ }
  \label{fig:breaking}
\end{figure}

Applying the same arguments as would prove Lemma \ref{lem:regularity_compactness_2_inputs}, we conclude
\begin{lem} \label{lem:boundary_moduli_spaces_a_infty}
The moduli spaces $\Discbar_{d}(x_0; \vx) $  are compact,  and are empty for all but finitely many $x_0$ once the inputs $\vx$ are fixed.  For a generic choice  $\bfD_{\mu}$, they form manifolds of dimension
\begin{equation*}
  |x_0| + d - 2 - \sum_{1 \leq k \leq d} |x_k|
\end{equation*}
whose boundary is covered by the images of the inclusions \eqref{eq:codim_1_strata_discs_1_output}.
 \noproof
\end{lem}

Whenever $|x_0| =  2 - d + \sum_{1 \leq k \leq d} |x^k|  $,  there are therefore only finitely many elements of   $\Disc_{d}(x_0; \vx) $.  Via the procedure described in Appendix \ref{sec:signs}, every such element $u \co S \to M $ induces an isomorphism
\begin{equation}
  \ro_{\psi^{w_{d,S}} x_d} \otimes \cdots \otimes \ro_{\psi^{w_{1,S}} x_1} \to \ro_{\psi^{w_{0,S}} x_0}.
\end{equation}
We write $\mu_{u}$ for the induced map on orientation lines, and omitting composition with $CW( \psi^{w_{k,S}} )  $ or its inverse from the notation, we define the $d$\th higher product 
\begin{equation*}
  \mu^{\F}_d \co  CW^{*}_{b}(L_{d-1}, L_d)    \otimes \cdots \otimes    CW^{*}_{b}(L_1, L_2)  \otimes  CW^{*}_{b}(L_0,L_1 )  \to   CW^{*}_{b}( L_0, L_d)
\end{equation*}
as a sum
\begin{equation}
  \mu^{\F}_{d}([x_d], \ldots, [x_1])  = \sum_{\substack{|x_0| =  2 - d + \sum_{1 \leq k \leq d} |x_k| \\ u \in \Disc_{d}(x_0; \vx) }} (-1)^{\dagger}  \mu_{u}([x_d], \ldots,  [x_1])
\end{equation}
with sign given by the formula
\begin{equation} \label{eq:dagger_sign}
 \dagger = \sum_{k=1}^{d} k |x_k|.
\end{equation}
This is simply a complicated way of saying that we count the elements of $\Disc_{d}(x_0; \vx)   $  with appropriate signs.   

If we now consider the $1$ dimensional moduli spaces, Lemma \ref{lem:boundary_moduli_spaces_a_infty} asserts that their boundaries are given by the strata appearing in Equation \eqref{eq:codim_1_strata_discs_1_output} which are rigid and therefore correspond to the composition of operations $\mu^{\F}_d$.  If we take signs into account, we conclude:
\begin{prop} \label{prop:a_infty_structure}
The operations $\mu^{\F}_d$ define an $A_{\infty}$ structure on the category $\Wrap_{b}(M)$.  In particular, the sum
\begin{equation}
  \label{eq:a_infty_property}
   \sum_{\substack{d_1 + d_2 = d +1 \\ 0 \leq k <d_1}} (-1)^{\maltese_{1}^{k}} \mu^{\F}_{d_1}\left(x_d, \ldots, x_{k+d_2+1}, \mu^{\F}_{d_2}(x_{k+d_2}, \ldots, x_{k+1}) ,   x_k, \ldots , x_1 \right) = 0 
\end{equation}
vanishes, where the value of the sign is given by
\begin{equation*} \maltese_{1}^{k} = k + \sum_{1 \leq j \leq k} |x_j|. \end{equation*}
\end{prop}

\section{Construction of the functor} \label{sec:construction-functor}
In Section \ref{sec:at-level-objects}, we assigned to each Lagrangian in $\Wrap_{b}(M)$ an object $\cF(L)$ in $\Tw(\Moore(Q))$.  We shall now extend this assignment to an $A_{\infty}$ functor, and the first task is to construct a chain map
\begin{equation} \label{eq:linear_term}
  \cF^{1} \co CW^{*}_{b}(L_0, L_1)  \to \Hom_{*}(\cF(L_0), \cF(L_1) ) .
\end{equation}
\subsection{Moduli spaces of half-strips} \label{sec:moduli-spaces-half}
On the half strip $Z_{+} =   [0,+\infty) \times [0,1] $, let $\xi^{0} =   (0,1)$ and $\xi^{-1} =   (0,0)$, and consider the surface $T = Z_{+} - \{ \xi^{0}, \xi^{-1} \}$  drawn in a non-standard way in Figure \ref{fig:half-disc}.  If we think of $T$ as the complement of $3$ marked point in a disc, we write $\xi^{1}$ for the puncture at infinity; and call the segment between $\xi^{0}  $  and $\xi^{-1} $ outgoing.  Note that $T$ is biholomorphic to the surface $S$ shown in Figure \ref{fig:disc_2_inputs}, but we choose to use a different notation because maps from $T$ to $M$ will satisfy not satisfy the same Cauchy-Riemann equation we imposed on $S$.  The most important difference is that we shall consider a map
\begin{equation*}
  H_{T} \co T \to \sH(M)
\end{equation*}
which near $\xi^{1}$ agrees with $H$, and whose restriction to the interval between $\xi^0$ and $\xi^{-1}$ takes values in $\sH_{Q}(M)$   (i.e. vanishes on $Q$).
 
Writing $X_T$ for the Hamiltonian flow of $H_T$, we shall consider the Cauchy-Riemann equation
\begin{equation}
  \label{eq:dbar_half_strip}
  (du - X_{T} \otimes d \tau)^{0,1} = 0
\end{equation}
for maps $u \co T \to M $.   To specify the choice of a family of almost complex structure with respect to which we take the $(0,1)$ part, let us fix a positive strip-like end $\epsilon^0$ near $\xi^0$ and a negative end $\epsilon^{-1}$ near $\xi^{-1}$, and choose a map
\begin{equation*}
  I_{T} \co  T \to \sJ(M)
\end{equation*}
which agrees with $I_t$ near $\xi^{1}$, and whose pullback under $\epsilon^0$ and $\epsilon^1$ also agrees with $I_t$.  
\begin{lem} \label{lem:restriction_half_disc_equation_ends}
The restriction of Equation \eqref{eq:dbar_half_strip} agrees with Equation \eqref{eq:dbar_strip} near $\xi^1$, and its pullback under $\epsilon^0$ and $\epsilon^1$ agrees with Equation \eqref{eq:dbar_no_X}. \noproof
\end{lem}

\begin{figure}
  \centering
  \includegraphics{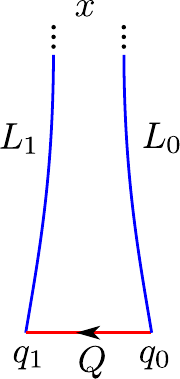}
  \caption{ }
  \label{fig:half-disc}
\end{figure}
Given a chord $x \in \Chord(L_0, L_1)$, and intersection points $q_0$ (respectively $q_1$) between $L_0$ (or $L_1$) and $Q$, it therefore makes sense to define $\Pil( q_0 , x , q_1 )  $ (see Figure \ref{fig:half-disc}) to be the moduli space of solutions $u$ to Equation \eqref{eq:dbar_half_strip} with boundary conditions
\begin{equation*}
\begin{cases} u(z) \in L_0 & \textrm{if $z \in \partial T$ lies on the segment between $\xi^0$ and $\xi^1$} \\
u(z) \in L_1 & \textrm{if $z \in \partial T$ lies on the segment between $\xi^{-1}$ and $\xi^1$} \\
u(z) \in Q & \textrm{if $z \in \partial T$ lies on the outgoing segment.}
\end{cases}
\end{equation*}
and asymptotic conditions
\begin{align*}
\lim_{s \to +\infty}  u(s, \cdot)  & = x(\cdot) \\
  \lim_{s \rightarrow - \infty} u \circ \epsilon^{1} (s,t)&  = q_1  \\ 
  \lim_{s \rightarrow + \infty} u \circ \epsilon^{0}(s,t) &  = q_0.
\end{align*}
From the standard transversality results in Floer theory we conclude:
\begin{lem}
For generic choices of Hamiltonians $H_T$ and almost complex structure $J_T$, the moduli space $ \Pil( q_0 , x , q_1 )   $  is a smooth manifold of dimension
\begin{equation}   |q_0| - |q_1| -  |x|. \end{equation} \noproof
\end{lem}

We assume that a choice of $H_t$ and $J_T$ is now fixed, and proceed to analyse the boundary of the Gromov compactification.  Since $T$ is topologically a $3$-punctured discs, there is no modulus, and the only strata that we need to add to the Gromov compactification are obtained by considering  breakings of strips at the three ends.  By Lemma \ref{lem:restriction_half_disc_equation_ends}, the virtual codimension $1$ strata (see Figure \ref{fig:breaking_half_discs}) are
\begin{align} \label{eq:half_disc_top_break}
& \coprod_{x_0 \in \Chord(L_0,L_1)}  \Pil( q_0 , x_0 , q_1 )  \times \Disc(x_0,x)  \\ \label{eq:half_disc_right_break}
& \coprod_{q'_1 \in L_1 \cap Q } \Pil( q_0 , x , q'_1 )  \times \Pil(q'_1, q_1)  \\ \label{eq:half_disc_left_break}
&   \coprod_{q'_0 \in L_0 \cap Q } \Pil(q_0, q'_0) \times   \Pil( q'_0 , x , q_1 ).
\end{align}

\begin{figure}
  \centering
  \includegraphics{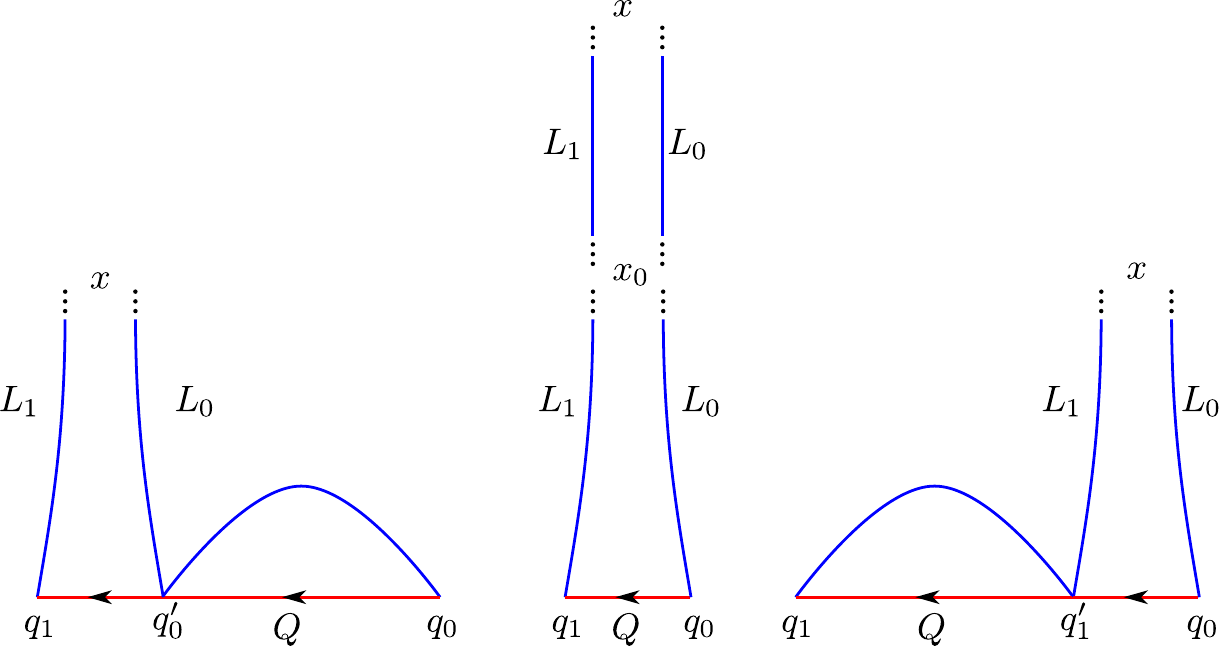}
  \caption{ }
  \label{fig:breaking_half_discs}
\end{figure}

The standard compactness result implies
\begin{lem}
The Gromov compactification $  \Pilbar( q_0 , x , q_1 ) $ is a compact manifold whose boundary is covered by the images of the strata \eqref{eq:half_disc_top_break}-~\eqref{eq:half_disc_left_break}.  \noproof
\end{lem}

In addition, we shall need to compare the product orientations on the strata of $ \partial \Pilbar( q_0 , x , q_1 )  $.  The proof is delayed until Appendix \ref{sec:signs}.
\begin{lem} \label{lem:orientaiton_error_moduli_half_disc_1_puncture}
The difference between the orientation induced as a boundary of $ \Pilbar( q_0 , x , q_1 ) $ and the product orientation is given by signs whose parity is
\begin{align*}
|x_0| + |q_0| & \textrm{ for the strata \eqref{eq:half_disc_top_break}if  $\Disc(x_0,x)$ is rigid.} \\
1+ |q_0| + |x| + | q'_1| &  \textrm{ for the strata \eqref{eq:half_disc_right_break}} \\
0 &  \textrm{ for the strata \eqref{eq:half_disc_left_break}.}
\end{align*}

\end{lem}

\subsection{The linear term}
Consider the evaluation map
\begin{equation} \ev \co \Pilbar(q_0, x,q_1) \to \Omega(q_0,q_1) \end{equation}
which takes every half disc $u$ to the path along $Q$ between $q_0$ and $q_1 $ obtained by restricting $u$ to the outgoing segment.  As in Section \ref{sec:at-level-objects}, we shall use the length parametrisation in order to remove any ambiguity.  In particular, if we consider the boundary strata \eqref{eq:half_disc_top_break}, we obtain a commutative diagram
\begin{equation*}
 \xymatrix{  \Pil( q_0 , x_0 , q_1 )  \times \Disc(x_0,x)   \ar[r] \ar[d]^{\ev} &  \Pilbar( q_0 , x , q_1 )    \ar[d]^{\ev} \\
  \Pil( q_0 , x_0 , q_1 )   \ar[r] &   \Omega(q_0,q_1) & } \end{equation*}
in which the left vertical arrow is projection to the first factor and the top map is the inclusion of a boundary stratum.    While if we consider the strata \eqref{eq:half_disc_right_break} and \eqref{eq:half_disc_left_break}, we obtain diagrams
\begin{equation*}
 \xymatrix{  \Pil( q_0 , x , q'_1 )  \times \Pil(q'_1, q_1)   \ar[r] \ar[d] &  \Pilbar( q_0 , x , q_1 )    \ar[d] & \ar[l] \ar[d]  \Pil(q_0, q'_0) \times   \Pil( q'_0 , x , q_1 )  \\
   \Omega(q_0,q'_1) \times  \Omega(q'_1,q_1) \ar[r]  & \Omega(q_0,q_1) & \ar[l]  \Omega(q_0, q'_0) \times   \Omega( q'_0  , q_1 )   . }
\end{equation*}
where the arrows in the bottom row are obtained by concatenation.

\begin{lem} \label{lem:fundamental_chains_linear_map}
There exist fundamental chains $ [ \Pilbar(q_0, x,q_1)] \in C_{*} (  \Pilbar(q_0, x,q_1) )  $  such that the assignment
\begin{equation}
\cF^{1}( [x] ) = \bigoplus_{q_0, q_1}  (-1)^{|x| + (|q_0|+1)(|x| + |q_1|  )}  \ev_{*} ( [ \Pilbar(q_0, x,q_1)]) 
\end{equation}
defines the chain map described in Equation \eqref{eq:linear_term}.
\end{lem}
\begin{rem}
More precisely, the choice of fundamental chain depends up to sign on the element $ [x] $.   As explained in Appendix \ref{sec:orientations}, the moduli space $ \Pilbar(q_0, x, q_{1}) $ is oriented \emph{relative} the line $\ro_{x}$, which means that an orientation of $\ro_{x}$ induces an orientation of the moduli space, and the wrapped Floer complex is precisely generated by such choices of orientations.
\end{rem}

The strategy is to start with the fundamental chains chosen for the moduli spaces $  \Pilbar(q_i, q_j) $ in Section \ref{sec:at-level-objects}, and to follow the procedure described in that section to choose fundamental chains for the moduli spaces $\Discbar(x_0,x)$.    Next, we choose fundamental cycles for those moduli spaces $\Pilbar(q_0, x,q_1) $  which do not have boundary compatibly with orientations.

Let us consider moduli spaces $\Pilbar(q_0, x,q_1) $ whose boundary only has codimension $1$ strata, which must therefore be products of closed manifolds.  By taking the product of the fundamental chains of the factors in each boundary stratum, we obtain a chain in $C_{*} (\Pilbar(q_0, x,q_1) )$
\begin{multline} \label{eq:sum_boundary_classes}
      \sum_{x_0 \in \Chord(L_0,L_1)}  \pm  [\Pil( q_0 , x_0 , q_1 )]  \times [\Disc(x_0,x) ] +   \sum_{q'_1 \in L_1 \cap Q } [\Pil(q_0, q'_0)] \times   [\Pil( q'_0 , x , q_1 )]   \\ 
+\sum_{q'_0 \in L_0 \cap Q } (-1)^{  1+ |q_0| + |x| + | q'_1|}   [\Pil( q_0 , x , q'_1 )]  \times [\Pil(q'_1, q_1)]
\end{multline}
where the sign in the first summation is $   |x_0| + |q_0| $ if  $\Disc(x_0,x)$ is rigid, and will not enter in any of our constructions otherwise.

Lemma \ref{eq:half_disc_right_break} implies that Equation \eqref{eq:sum_boundary_classes} is a cycle.  We now define the fundamental chain
\begin{equation*}
  [\Pil( q_0 , x , q_1 )] 
\end{equation*}
to be any chain whose boundary is Equation \eqref{eq:sum_boundary_classes}.

\begin{proof}[Proof of Lemma \ref{lem:fundamental_chains_linear_map}]
Consider one of the strata of $\partial \Pilbar(q_0, x,q_1)  $ listed in Equation \eqref{eq:half_disc_top_break}, for which the moduli space $ \Disc(x_0,x) $ is not rigid.  Since the fundamental chain of this stratum is obtained by taking the products of the fundamental chains on each factor, its image in $C_{*}( \Omega(q_0,q_1) )  $ is degenerate, and hence vanishes.  We conclude that such strata do not contribute to $ \partial \ev_{*}(  [\Pilbar(q_0, x,q_1) ])  $.

 Consulting Equation \eqref{eq:differential_twisted} and the definition of $\cF(L)$ in Lemma \ref{lem:definition_twisted_complex_Lag}, we find that, up to signs which we shall ignore, proving that $\cF^{1}$ is a chain map is equivalent  to proving that
\begin{multline} \label{eq:chain_map_equation}
  \mu_{1}^{\P} (  \cF^{1}([x]) ) + \mu_{2}^{\P}\left(\cF^{1}[x] ,  \sum_{q_{0}, q'_{0}} [\Pilbar(q_{0}, q'_{0})] \right) +  \mu_{2}^{\P}\left( \sum_{q_{1}, q'_{1}} [\Pilbar(q'_{1}, q_{1})], \cF^{1} [x] \right) \\ =  \cF^{1}(\mu^{\F}_1 [x]).
\end{multline}
Since it follows essentially by definition that
\begin{equation*}  \partial  \ev_{*}( [\Pilbar(q_0, x,q_1) ]) = \mu_{1}^{\P} (  \cF^{1}([x]) ) , \end{equation*}
we shall now interpret $  \ev_{*}( \partial  [\Pilbar(q_0, x,q_1) ]) $  differently to account for the remaining terms in the equation for a chain map.

Going through the stratification of $ \partial \Pilbar(q_0, x,q_1)  $, we find that the stratum \eqref{eq:half_disc_top_break} for rigid moduli spaces correspond to $\cF^{1}(\mu^{\F}_1 [x])$, and that the strata \eqref{eq:half_disc_left_break} and \eqref{eq:half_disc_right_break} respectively correspond to the second and third term in the left hand side of Equation  \eqref{eq:chain_map_equation}.
\end{proof}

\subsection{Homotopy between the compositions}
Having constructed the chain map $\cF^{1}$, we would need to check that it respects the product structure which on the source is given by the count of pairs of pants, and on the target is given by concatenation of paths.  As with the failure of associativity on the pair of pants product in the Fukaya category, it is only the map induced by $\cF^{1}$  on cohomology that respects the product structure.  At the chain level, there is a homotopy between the $\cF^{1}(\mu^{\F}_{2}([x_2], [x_1])) $  and $\mu_{2}^{\Tw(\Moore)}(\cF^{1} [x_2], \cF^{1} [x_1]  )  $; these two compositions are represented by the two outermost diagram in Figure \ref{fig:breaking_half_disc_two_inputs}.
\begin{figure}
  \centering
  \includegraphics{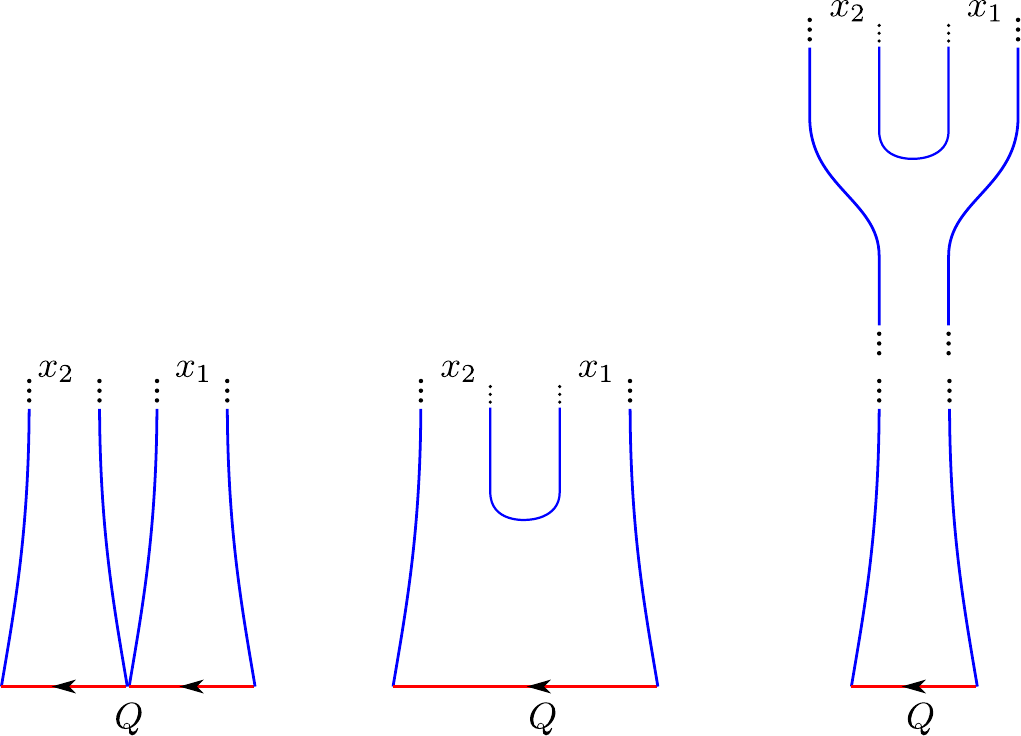}
  \caption{ }
  \label{fig:breaking_half_disc_two_inputs}
\end{figure}

We shall therefore have to introduce a moduli space $ \Pil_{3} $ which is $1$-dimensional, and whose boundary is represented by the two broken curves in Figure \ref{fig:breaking_half_disc_two_inputs}.    As in the proof of the homotopy associativity of the product in the Fukaya category, we shall define a family of Cauchy-Riemann equations on this abstract moduli space, interpolating between the equations on the two broken curves, and the moduli space of solutions to this equation, with appropriate boundary conditions, shall define the desired homotopy.

Since the notion of an $A_{\infty}$ functor requires the construction of a tower of such homotopies, we proceed with describing the construction in the general case:

\subsection{Abstract moduli spaces of half-discs}
Let us write $\Pil_{d}$ for the moduli space of holomorphic discs with $d+2$ marked points, of which $d$ successive ones are distinguished as incoming; the remaining marked points and the segment connecting them are called outgoing.  We identify $\Pil_{0}$ with a point (equipped with a group of automorphisms isomorphic to $\bR$) corresponding to the moduli space of strips. In addition, we fix an orientation on the moduli space $\Pil_{d}$ using the conventions for Stasheff polyhedra in \cite{seidel-book}, and the isomorphism
\begin{equation} \label{eq:half_discs_are_discs} \Pil_{d} \cong \Disc_{d+1}  \end{equation}
taking the incoming marked points on the source to the first $d$ incoming marked point on the target.   In particular, the usual Deligne-Mumford compactification of $\Disc_{d+1}$ yields a compactification of $  \Pil_{d}  $ by adding broken curves.  It shall be important, however, to understand the operadic meaning of the boundary strata of $  \Pilbar_{d} $. 

 If breaking takes place away from the outgoing segment, the topological type of the broken curve is determined by sequences $\{ 1 , \ldots, d_1 \}$ and $\{ 1, \ldots, d_2\}$, such that $d_1 +d_2 = d+1$, and a fixed element $k$ in the first sequence.  These data determine a map
\begin{equation}\label{eq:half_popsicle_break_top} \Pilbar_{d_1} \times \Discbar_{d_2} \to \Pilbar_{d}   \end{equation}
and the simplest such curve is shown in the right of Figure \ref{fig:breaking_half_disc_two_inputs}.

If breaking occurs on the outgoing segment, it is determined by a partition $\{1, \ldots, d \} = \{1, \ldots, d_1 \} \cup \{ d_1 +1, \ldots, d \}$.   Letting $d_2 = d - d_1$, we obtain a map
\begin{equation}\label{eq:half_popsicle_break_middle} \Pilbar_{d_1} \times  \Pilbar_{d_2} \to \Pilbar_{d} .  \end{equation}

\subsection{Floer data for half-discs}
The following definition is the direct extension of Definition \ref{def:floer_datum_disc_1_output} to the moduli space of half-discs.  Note that a universal choice of strip-like ends on the moduli spaces $\Disc_{d}$ induces one on $\Pil_{d-1}$ via the identification of Equation \eqref{eq:half_discs_are_discs}:
\begin{defin} \label{def:floer_datum_half_popsicle}
A  \emph{Floer datum} $D_{T}$ on a stable disc $T \in \Pilbar_{d}$  consists of the following choices on each component:
\begin{enumerate}
\item Weights: A positive real number $w_{k,T}$ associated to each end of $T$ which is assumed for be $1$ for $k= -1,0$.
\item Time shifting maps:  A map $ \rho_{T} \co \partial \bar{T} \to [1,+ \infty)$ which agrees with $w_{k,T}$ near $\xi^{k}$.
\item Hamiltonian perturbation: A map $H_{T} \co T \to \sH(M)$  on each surface such that the restriction of $H_{T}$ to a neighbourhood of the outgoing boundary segment takes value in $ \sH_{Q}(M) $, and whose value near the $\xi^k$ for $1 \leq k \leq d$ is
  \begin{equation*}
     \frac{H \circ \psi^{w_{k,T}}}{w^{2}_{k,T}}
  \end{equation*}
\item Basic $1$-form:  A closed $1$-form $\alpha_{T}$  whose restriction to the complement of the outgoing segment in $\partial T$ and to a neighbourhood of  $\xi^0$ and $\xi^{-1}$  vanishes, and whose pullback under  $ \epsilon^k $ for $1 \leq k \leq d$ agrees with  $ w_{k,T} d \tau $.
\item Almost complex structure:  A map $I_{T} \co T \to \sJ(M)$ whose pullback under $ \epsilon^k $ agrees with  $ (\psi^{w_{k,T}})^{*} I _t$.
\end{enumerate}
\end{defin}

Note that the closedness of $\alpha_{T}$ and the fact that it vanishes near the outgoing marked points implies that it does not vanish on the outgoing boundary segment.  Were it not for the fact that the outgoing segment is mapped to a \emph{compact} Lagrangian, this would cause a problem with compactness.  

As before, we write  $X_{T}$  for the Hamiltonian flow of $H_T$ and consider the differential equation
\begin{equation}
  \label{eq:dbar_half_discs}
  \left(du - X_{T} \otimes \alpha_{T}\right)^{0,1} = 0
\end{equation}
with respect to the $T$-dependent almost complex structure $I_{T}$.

\begin{lem}
The pullback of Equation \eqref{eq:dbar_half_strip} under the end $\epsilon^{k}$ agrees with Equation \eqref{eq:dbar_no_X} if $k = -1, 0$, and otherwise with
\begin{equation*}  
\left(du \circ \xi^{k} - X_{ \frac{H \circ \psi^{w_{k,T}}}{w_{k,T}} } \otimes d \tau \right)^{0,1} = 0.
\end{equation*} \noproof
\end{lem}

In order for the count of solutions to Equation \eqref{eq:dbar_half_discs} to define the desired homotopies, we must choose its restriction to the boundary strata of the moduli spaces in a compatible way.  In the case where $d=2$, the moduli space is an interval, whose two endpoints may be identified with the products $\Pil_{1} \times \Pil_{1} $  and $ \Pil_{1} \times \Disc_{2} $ (see Figure \ref{fig:breaking_half_disc_two_inputs}).    Our discussions in Sections \ref{sec:comp-wrapp-fukaya} and \ref{sec:moduli-spaces-half} fix Floer data respectively on the unique elements of $ \Disc_{2}  $ and  $\Pil_{1}$.    In the case of $ \Pil_{1} \times \Pil_{1} $ the equations on both sides of the node agree, in the coordinates coming form strip-like ends, with Equation  \eqref{eq:dbar_half_strip}, so we may glue the corresponding Floer data.

On the other hand, the equations on the two sides of the node of the broken curve representing an element of $ \Pil_{1} \times \Disc_{2} $ only agree up to applying the Liouville flow, since the weight on the input is $1$ while the weight of the output is $2$.  Our conventions that the weights $w_{0,T} $ and $w_{-1,T}$ are always equal to $1$ imply that we must apply the Liouville flow $\psi^{1/2}$ to the equation imposed on  $ \Disc_{2} $ in order for gluing to make sense.

In this way, we obtain Floer data in a neighbourhood of $\partial \Pilbar_{2} $, which may be extended to the interior of the moduli space.  Having fixed this choice, we proceed inductively for the rest of the moduli spaces:
\begin{defin}
A   \emph{universal and conformally consistent}  choice of Floer data for the homomorphism $\cF$,  is a choice $\bfD_{\cF}$  of such Floer data  for every integer $d \geq 1$, and every (representative) of an element of  $\Pilbar_{d}$ which varies smoothly over this compactified moduli space, whose restriction to a boundary stratum is conformally equivalent to the product of Floer data coming from either $\bfD_{\mu}$ or a lower dimensional moduli space  $\Pilbar_{d}$, and which near such a boundary stratum agrees to infinite order with the Floer data obtained by gluing.
\end{defin}
The consistency condition implies that each irreducible component of a curve representing a point in the stratum \eqref{eq:half_popsicle_break_top} carries the restriction of the datum $ \bfD_{\mu} $if it comes from the factor $ \Discbar_{d_2} $, and the restriction of the datum $\bfD_{\cF}$  if it comes from $ \Pilbar_{d_1} $, up to  conformal equivalence which  is fixed near by the requirement that $w_{0,T}$  and $w_{-1,T}$ are both equal to $1$. 

\subsection{Moduli spaces of half-discs} 
Given a sequence of chords $\vx = \{ x_{k} \in \Chord(L_{k-1},L_k) \} $ if $1 \leq k \leq d  $, and a pair of intersection points $q_0 \in L_0 \cap Q $  and $q_d \in L_d \cap Q$, we define $ \Pil(q_0, \vx, q_d ) $ to be the moduli space of solutions to Equation \eqref{eq:dbar_half_discs} whose source is an arbitrary element $T \in \Pil_{d}$, with boundary conditions 
\begin{equation*}
\begin{cases} u(z) \in \psi^{\rho(z)} ( L_0 )  & \textrm{if $z \in \partial T$ lies on the segment between $\xi^0$ and $\xi^1$} \\
u(z) \in \psi^{\rho(z)} (  L_k) & \textrm{if $1 \leq k < d$ and $z \in \partial T$ lies on the segment between $\xi^{k}$ and $\xi^{k+1}$} \\
u(z) \in \psi^{\rho(z)} (  L_d) & \textrm{if $z \in \partial T$ lies on the segment between $\xi^{-1}$ and $\xi^d$} \\
u(z) \in Q & \textrm{if $z \in \partial T$ lies on the outgoing segment.}
\end{cases}
\end{equation*}
and with asymptotic conditions
\begin{align*}
    \lim_{s \rightarrow + \infty} u(\epsilon^k(s,t )) &= x_k(t)  \textrm{ if $1 \leq k \leq d$} \\
 \lim_{s \rightarrow \infty} u(\epsilon^0(s,t )) & = q_0 \\
 \lim_{s \rightarrow - \infty} u(\epsilon^1(s,t )) & = q_d.
\end{align*}

The standard Sard-Smale type argument implies:
\begin{lem}
For generic choices of universal Floer data, $\Pil(q_0, \vx, q_d)$ is a smooth manifold of dimension
\begin{equation}  d  - 1 + |q_0|- |q_d| -  \sum_{k=1}^{d} |x_k|. \end{equation} \noproof
\end{lem}

 Moreover, the compactification of the abstract moduli space of  half-discs extends to a Gromov compactification of the moduli spaces  $\Pil(q_0, \vx, q_d)$.  The top strata of the boundary have a description analogous to the two types of boundary strata of the abstract moduli spaces;  assume in the first case  that $d_1 + d_2 = d+1$, that $k$ is an integer between $0$ and $d-d_2$.  Consider sequences of chords   $\vx[2]= (x_{k+1}, \ldots, x_{k+d_2}) $ and $\vx[1] = (x_1, \ldots, x_k, y, x_{k+d_2+1}, \ldots, x_{d})$ with $y \in \Chord(L_{k}, L_{k+d_2})$.   By gluing a disc with inputs $\vx[2]$ and output $y$ to a half-disc with inputs $\vx[1]$, we obtain a map
\begin{equation} \label{eq:moduli_space_maps_breaking_top} \Pilbar (q_0, \vx[1], q_d)  \times \Discbar(y,\vx[2]) \to   \Pilbar(q_0, \vx, q_d) \end{equation}

Similarly, given a partition of the inputs  $ \vx[1] = \{ x_1, \ldots, x_{d_1} \}$ and  $\vx[2] = \{x_{d_1+1}, \ldots, x_d  \} $, and an intersection point $q'_{d_1}$ between $L_{d_1}$ and $Q$, we obtain a map
\begin{equation}  \label{eq:moduli_space_maps_breaking_side} \Pilbar( q_0, \vx[1], q'_{d_1}) \times  \Pilbar( q'_{d_1} , \vx[2], q_d) \to  \Pilbar (q_0, \vx, q_d).  \end{equation}

Using the parametrisation by arc length as in Lemma \ref{lem:evaluate_discs_paths}, we observe the existence of compatible families of evaluation maps to spaces of paths in $Q$:
\begin{lem}
There exists a choice of parametrisations of the outgoing boundary segment of half-discs which yields maps
\begin{equation} \Pilbar(q_0, \vx, q_d) \to \Omega(q_0,q_d) \end{equation}
such that in the setting Equation \eqref{eq:moduli_space_maps_breaking_top} we have a commutative diagram
\begin{equation*}
 \xymatrix{\Pilbar (q_0, \vx[1], q_d)  \times \Discbar(y,\vx[2])  \ar[r] \ar[d] &  \Pilbar(q_0, \vx, q_d)  \ar[d]  \\
 \Pilbar (q_0, \vx[1], q_d)  \ar[r] &  \Omega(q_0,q_d)  } \end{equation*}
while in the setting of Equation \eqref{eq:moduli_space_maps_breaking_side}  the following diagram also commutes
\begin{equation*}
 \xymatrix{  \Pilbar( q_0, \vx[1], q'_{d_1}) \times  \Pilbar( q'_{d_1} , \vx[2], q_d)    \ar[r] \ar[d] & \Pilbar(q_0, \vx, q_d)    \ar[d] \\
 \Omega( q_0, q'_{d_1}) \times  \Omega( q'_{d_1} , q_d)   \ar[r]  &  \Omega(q_0,q_d) . }
\end{equation*} \noproof
\end{lem}

The proof of the next result again follows from standard gluing techniques which are reviewed in Appendix \ref{app:manifold-with-boundary}. The fact that we are using Hamiltonians which are not linear along the infinite cone might a priori imply that we lose compactness.  However, the argument used in proving Lemma \ref{lem:gromov_compactness} applies here as well.

\begin{lem} \label{lem:topological_manifold_boundary_half_popsicle}
The moduli space $\Pilbar(q_0, \vx, q_d)$ is a compact topological manifold with boundary stratified by smooth manifolds.  The codimension $1$ strata are the images of the inclusion maps \eqref{eq:moduli_space_maps_breaking_top},  \eqref{eq:moduli_space_maps_breaking_side}. \noproof
\end{lem}

\subsection{Compatible choices of fundamental chains}
In order to define operations using the moduli space $\Pilbar(q_0, \vx, q_d)  $ we must compare the orientation of the boundary of this moduli space with that of the interior.  As with other sign-related issues, the proof of this result is delayed to Appendix \ref{sec:signs}:
\begin{lem} \label{lem:product_orientation_half_discs}
The product orientation on the strata \eqref{eq:moduli_space_maps_breaking_side} differs from the boundary orientation by
\begin{equation}  \label{eq:sign_flat} \flat = (d_2+1)\left(|q_0| + |q'_{d_1}|+  \sum_{i=1}^{d_1} | x_{i}| \right) + d_1 +1 , \end{equation}
while the difference between the two orientations on the stratum \eqref{eq:moduli_space_maps_breaking_top} is given by
\begin{equation} \label{eq:sign_sharp} \sharp=  d_2( |q_0| + \sum_{j=1}^{k+d_2} |x_j|  ) + d_{2}(d-k) + k +1   \end{equation}
whenever $\Discbar(x,\vx[2]) $  is rigid  and  $x$ is the $k+1$\st element of $\vx[1]$. 
\end{lem}

Following the same strategy as for the proof of Lemma \ref{lem:fundamental_chains_linear_map}, we may equip these moduli spaces with appropriate fundamental chains in the cubical theory:
\begin{lem} \label{lem:good_choice_fundamental}
There exists a family of fundamental chains
\begin{equation} [\Pilbar(q_0, \vx, q_d)] \in C_{*}(\Pilbar(q_0, \vx, q_d)  )  \end{equation}
such that
\begin{multline}  \label{eq:boundary_fundamental_chain} \partial [ \Pilbar(q_0, \vx, q_d)] = \sum_{\substack{ 0 \leq d_1 \leq  d  \\  q'_{d_1} \in L_{d_1} \cap Q } } (-1)^{\flat}   [   \Pilbar( q_0, \vx[1], q'_{d_1})] \times  [  \Pilbar( q'_{d_1} , \vx[2], q_d)] + \\  \sum_{y \in  \Chord(L_{k}, L_{k+d_2}) } (-1)^{\sharp} [ \Pilbar (q_0, \vx[1], q_d)  ]  \times [\Discbar(y,\vx[2])]  . \end{multline} \noproof
\end{lem}

\subsection{Definition of the functor}

We now define a map
\begin{align} 
\cF^{d} \co  CW^{*}_{b}(L_{d-1}, L_d)    \otimes \cdots \otimes    CW^{*}_{b}(L_1, L_2)  \otimes  CW^{*}_{b}(L_0,L_1 )    & \to \Hom_{*}(\cF(L_0), \cF(L_d)) \\
[ x_d ]\otimes \cdots \otimes [x_1] & \to \bigoplus_{\substack{q_0 \in L_0 \cap Q \\ q_d \in L_d \cap Q}} (-1)^{\ddagger} \ev_{*} ( [\Pilbar_{\bfp}(q_0, \vx, q_d)]) .
\end{align}
where the sign is given by 
\begin{align*}
\ddagger & = \sum_{1 \leq k \leq d} k |x_k| + (d+1)|q_d|   + (|q_0| + d ) \dim\left(  \Pil(q_0, \vx, q_d)\right)
\end{align*}
\begin{lem}
The collection of maps $\cF^{d}$ satisfy the $A_{\infty}$ equation for functors \begin{equation} \label{eq:a_infty_functor_equation} \sum_{d_1+d_2 = d+1} (-1)^{\maltese_i} \cF^{d}( \id^{d_1 - i} \otimes \mu^{\F}_{d_2} \otimes \id^{i}) = 
\mu_{1}^{\Tw(\Moore)}  \cF^{d} + \sum_{d_1 + d_2 = d} \mu_{2}^{\Tw(\Moore)}(\cF^{d_2} , \cF^{d_1}).   \end{equation}
\end{lem}
\begin{proof}
The discussion of signs is omitted as it is tedious, and does not differ in any significant way from the computation performed in Section (9f) of \cite{abouzaid-seidel}. The correspondence between the terms of the $A_{\infty}$ equation \eqref{eq:a_infty_functor_equation} and expressions for the boundary of the fundamental chain \eqref{eq:boundary_fundamental_chain} is particularly simple:  (i) the left hand side of Equation  \eqref{eq:boundary_fundamental_chain} corresponds to the first term on the right hand side of Equation  \eqref{eq:a_infty_functor_equation} (ii) the first sum on the right hand side of Equation  \eqref{eq:boundary_fundamental_chain} corresponds  to the summation on the right hand side of  \eqref{eq:a_infty_functor_equation}, and (iii) the right hand side in Equation  \eqref{eq:a_infty_functor_equation} cancels with the remaining terms from the right hand side of Equation  \eqref{eq:boundary_fundamental_chain}.   Note that, a priori, the last  sum in Equation \eqref{eq:boundary_fundamental_chain}  has contributions from moduli spaces of discs which are not necessarily rigid; however, the image of such contribution under the evaluation map to $\Moore(Q)$ vanishes as all such chains become degenerate (i.e., factor through a lower-dimensional chain). 
\end{proof}

\section{Equivalence for cotangent fibres}  \label{sec:cotangent}
Let us now equip $Q$ with a Riemannian metric.  The cotangent bundle of $Q$ is an exact symplectic manifold with primitive $\lambda = \sum p_i dq_i$ in local coordinates, and the points of $p$-norm less than $1$ form  an exact symplectic manifold $D^* Q$ with contact type boundary.  In particular, we have a decomposition of the cotangent bundle as a union
\begin{equation} S^*Q \times [1,+\infty) \cup D^* Q \end{equation}
with $S^*Q \times 1$ the contact boundary of $D^* Q$.  Let us pick a function $H$ which agrees with $|p|^{2}$ on the infinite cone.  Choosing the metric   generically, we may ensure that such a Hamiltonian  has only non-degenerate chords with end points on $T^*_{q} Q$.

Abbondandolo and Schwarz have already proved in \cite{AS} that there is an isomorphism
\begin{equation} \label{eq:AS-map} \Theta \co  H_{*}(\Omega(q,q)) \to HW^{*}_{b}(T^*_{q} Q, T^*_{q} Q) \end{equation}
using a Morse model for the left hand side.   In this section, we shall sketch the proof of
\begin{lem} \label{eq:left_inverse}  The linear term 
\begin{equation} \cF^{1} \co CW^{*}_{b}(T^*_{q} Q, T^*_q Q ) \to C_{*}(\Omega(q,q)) \end{equation}
of the functor $\cF$ descends on cohomology to an inverse for $\Theta$.   In particular, $\cF$ is an equivalence.
\end{lem}
\begin{rem}
Since the conventions of Abbondandolo and Schwarz  for orientations on the Floer side are the same,  up to chain equivalence, to those we use in Appendix \ref{sec:orientations} for a trivial background class,  the result of \cite{AS} suffers from the sign issue alluded to in Remark \ref{rem:sign_error}.   The map constructed in  \cite{AS} becomes a chain map if we consider the Morse homology of the based loop space with values in this local system.  Alternatively, one can change the sign conventions on the Floer side by twisting its generators by $ \kappa_{TQ}  $, and obtain a chain map from a twisted version of the wrapped Floer cohomology to the ordinary homology of the based loop space.  Lemma \ref{lem:twisted_to_untwisted_sign_difference} implies that the resulting signs agree with those coming from considering the cotangent fibre as an object of the wrapped Fukaya category with the background class $b$ which is the pullback of $w_{2}(Q)$.  This allows us to use the results of \cite{AS} whenever $Q$ is orientable.

If $Q$ is not orientable then the pullback of $TQ$ under a based loop may not be trivial.  However, it admits two $\Pin$ structure and $\kappa_{TQ}$ is defined in this case to be the vector space generated by the two $\Pin$ structures with the relation that their sum vanishes.  In this way, the results of \cite{AS} concerning the based loop space extend to the non-orientable case as noted in Section 5.9 of \cite{AS} (though the relation between symplectic cohomology and the homology of the free loop space requires twisting the latter by the orientation line of the base).
\end{rem}
Since $\Theta$ is known to be an isomorphism, it suffices to prove that $\cF$ is a one-sided inverse; we shall prove that it is a right inverse.  In order to pass to Morse chains, we follow \cite{AS}, and fix a Morse function $\LAction$ on 
\begin{equation} \Omega^{1,2}(q,q) \equiv W^{1,2}(([0,1],0,1) , (Q,q,q)) \end{equation}
which is the Lagrangian action functional for the Legendre transform of $H$.  The critical points of $\LAction$ are non-degenerate by virtue of the analogous result for the critical points of $\Action$, so we obtain a Morse-Smale flow on $\Omega^{1,2}(q,q)$ upon choosing a sufficiently generic metric.

Let us write
\begin{equation} \label{eq:inclusion_weakly_transverse} C^{\tr}_{n}\left(\Omega^{1,2}(q,q) \right) \subset C_{n}\left(\Omega^{1,2}(q,q) \right) \end{equation}
for the  subgroup spanned by continuous chains of dimension $n$ which satisfy the following property:
\begin{equation} \parbox{36em}{the restriction to a boundary stratum of dimension $k$ is smooth near the inverse images of ascending manifolds of codimension $k$, and intersects the ascending manifolds of codimension larger than $k$ transversely}
\end{equation}
We call such chains \emph{weakly transverse}. Note that weak transversality implies that a chain of dimension $k$ does not intersect the ascending manifolds of critical points of index greater than $k$.
\begin{lem}
The inclusion of weakly transverse chains into all chains induces an isomorphism on cohomology.  Moreover, for appropriate choices of fundamental chains on the moduli spaces of half-popsicle maps, the map
\begin{equation} \label{eq:map_floer_to_transverse}  CF^{*}_{b}(T^*_{q} Q;H) \to  C_{*}\left(\Omega(q,q)) \right)  \end{equation}
factors through weakly transverse chains.
\end{lem}
\begin{proof}
The first part follows by considering a smaller subcomplex consisting of all smooth chains which are transverse to all ascending manifolds of $\LAction$.  As the inclusion of this subcomplex into $C_{*}\left(W^{1,2}(([0,1],0,1) , (Q,q,q)) \right)$ induces an isomorphism on cohomology, we conclude that the inclusion \eqref{eq:inclusion_weakly_transverse} must induce a surjection on cohomology groups.  To obtain an injection, we must show that a weakly transverse chain is homologous to a transverse one through weakly transverse chains.  This follows from the fact that weakly transversality is stable under perturbations which are $C^1$ small near the intersections with ascending manifolds of codimension $k$, since the result of a generic such perturbation is a homotopic transverse chain.

For the second part, we first observe that Theorem 3.8 of \cite{AS} implies that for a generic choice of almost complex structures the evaluation map
\begin{equation*} \Pil(q,x,q) \to \Omega^{1,2}(q,q)  \end{equation*}  
is transverse to ascending manifolds of $\LAction$ of codimension $|x|$ or higher for every $X$-chord $x$.  The inverse image of such ascending manifolds consists of finitely many points in $\Pil(q,x,q)$ by Gromov compactness.  By induction, we may therefore choose fundamental chains for $\Pilbar(q,x,q)$ satisfying the desired property.
\end{proof}

The main reason for introducing $C^{\tr}_{n}\left(\Omega^{1,2}(q,q) \right)$ is that the count of negative gradient flow lines starting on a given cell and ending on a critical point of $\LAction$ defines a chain equivalence
\begin{equation*} C^{\tr}_{n}\left(\Omega^{1,2}(q,q) \right) \to CM_*(\Omega^{1,2}(q,q), \LAction). \end{equation*}  

Composing with \eqref{eq:map_floer_to_transverse} we obtain a map
\begin{equation} \label{eq:inverse_to_AS}  CW^{*}_{b}(T^*_{q} Q, T^*_{q} Q) \to  CM_*(\Omega^{1,2}(q,q), \LAction) ,\end{equation}
where the right hand side is the Morse complex of the function $\LAction$.  

We end this section with a sketch of the proof of Lemma \ref{eq:left_inverse}.  In particular, we omit discussing signs.
\begin{proof}[Proof of Lemma \ref{eq:left_inverse}]
Abbondandolo and Schwarz construct a chain homomorphism
\begin{align*} \Theta \co CM_*(\Omega^{1,2}(q,q), \LAction) & \to CW^{*}_{b}(T^*_{q} Q;T^*_{q} Q)  \\
y & \mapsto \sum_{|x| = |y|} n^{+}(y;x) x
\end{align*}
where $n^{+}(y;x)$ is a count of maps
\begin{equation*} (-\infty, 0] \times   [0,1] \to T^* Q \end{equation*}
solving Equation \eqref{eq:dbar_strip} which converge at infinity to $x$ with boundary condition $T^*_{q}Q$ along the semi-infinite segment $ (-\infty, 0] \times  \{ 0 \}$  and $ (-\infty, 0] \times  \{ 1 \}$, and such that
\begin{equation*}  \parbox{36em}{the image of $u( \{ 0 \}\times  [0,1]  )$ under the projection from $T^*Q$ to $Q$ is a path on $Q$ lying in the descending manifold of $y$.}  \end{equation*}

We shall prove that the composition
\begin{equation} \label{eq:composition_functor_AS} CM_*(\Omega^{1,2}(q,q), \LAction)  \stackrel{\Theta}{\rightarrow}  CW^{*}_{b}(T^*_{q} Q,T^*_{q} Q) \stackrel{\cF^1}{ \rightarrow}  CM_*(\Omega^{1,2}(q,q), \LAction)  \end{equation}
is homotopic to the identity.  To this end, for each positive real number $a$ and critical points $y$ and $z$ of $\LAction$, we consider a moduli space
\begin{equation*} \Cobord(y,z;a) \end{equation*}
consisting of maps 
\begin{equation*} [0,a]  \times [0,1] \to T^* Q \end{equation*}
which solve the Cauchy-Riemann problem of Equation \eqref{eq:dbar_strip}, have  boundary condition $T^*_{q}Q$ along the segments  $ [0, a] \times  \{ 0 \}$  and $ [0, a] \times  \{ 1 \}$, and such that
\begin{equation*}  \parbox{36em}{$u( \{ 0 \} \times [0,1] )$ is contained in the zero section and, considered as a path on $Q$, lies on the ascending manifold of $z$, while the image of $u( \{ a \} \times [0,1] )$ under the projection from $T^*Q$ to $Q$ is a path on $Q$ lying in the descending manifold of $y$.}  \end{equation*}

For $a=0$, any such solution is necessarily constant, so that the count of rigid elements of $\Cobord(y,z;0)$ gives the identity map of $ CM_*(\Omega^{1,2}(q,q), \LAction) $.  Let $\Cobordbar(y;z)$ denote the Gromov compactification of the union of all moduli spaces $\Cobord(y,z;a)$ for $a \in [0,+\infty)$.  Letting $a$ go to $+\infty$, a family of maps $u_a$ defined on finite strips $ [0,a] \times [0,1] $  breaks into two maps $u_-$ and $u_+$ defined on semi-infinite strips.  An inspection of the boundary conditions shows that the count of rigid elements of such a boundary stratum reproduces the composition of $\cF^1$ with $\Theta$ as in Equation \eqref{eq:composition_functor_AS}.

The remaining boundary strata occur at finite $a$ when the projection of $u(\{ a \} \times [0,1] )$ to $Q$ escapes to the ascending manifold of a critical point $y'$ which differs from $y$. Similarly, the image of $u( \{ 0 \} \times [0,1] )$ may converge to the descending manifold of a critical point $z'$.  Such boundary strata are in bijective correspondence with
\begin{equation*} \Morse(y,y')  \times \Cobordbar(y';z)  \cup \Cobordbar(y;z') \times \Morse(z',z) \end{equation*}  
where $\Morse(y,y')$ is the moduli space of gradient trajectories of $\LAction$; the count of rigid element of these moduli spaces defines the differential on $ CM_*(\Omega^{1,2}(q,q), \LAction)  $.  In particular, we conclude that counting rigid elements of $\Cobordbar(y,z)$ is a homotopy between the identity and the composition of $\cF^1$ with $\Theta$ completing the proof of Lemma \ref{eq:left_inverse}.
\end{proof}

\appendix
\section{Technical results}
\subsection{Relative $\Pin$ structures and orientations} \label{sec:orientations}
Consider a pair of graded Lagrangians $L$ and $K$ in a symplectic manifold $M$, and a time-$1$ non-degenerate Hamiltonian chord $y$ starting at $L$ and ending on $K$.  Let $T$ be a genus $0$ Riemann surface obtained by removing a point from the boundary of $D^2$, and equipped with a negative end near this puncture.  We choose a projection map 
\begin{equation} \label{eq:project_disc_strip}
  t \co  T \to [0,1]
\end{equation}
which extends the projection to the interval along the strip-like end, and a $1$-form $\beta$ on $T$ which vanishes near the boundary and agrees with the pullback of $d \tau$ along the strip-like end.   Using the almost complex structure $I_{t}$ on the tangent space of $M$ allows us to define a Cauchy-Riemann operator
\begin{equation} \label{eq:determinant_operator_chord}
  \left( du - X \otimes \beta \right)^{0,1} =0
\end{equation}
on sections of $T \times M$.  The linearisation of this operator has source sections of $(y \circ t)^{*}(TM)$, and the data of gradings on $L$ and $K$ define a path of Lagrangians $\Lambda_y$ in the restriction of this symplectic vector bundle to the boundary, uniquely determined up to homotopy by the requirement that it lift to a path of graded Lagrangians interpolating between the graded structures on $T_{y(0)}L$ and $T_{y(1)}K$ (see Section (11) of \cite{seidel-book}) on either side of the puncture.    
\begin{defin}[Lemma 11.11 of \cite{seidel-book}]
The vector space $\ro_{y}$ is the determinant bundle of the linearisation of the operator \eqref{eq:determinant_operator_chord} with  boundary conditions $\Lambda_{y}$.
\end{defin}
\begin{rem}
If we were working in $T^* Q$, and we replace $L$ by the zero section and $K$ by the graph of the differential of a Morse function, then $\ro_y$ can be proved to be canonically isomorphic to the top exterior power of the descending manifold of the critical point associated to this Lagrangian intersection.
\end{rem}

The Floer chain complexes we consider are direct sums of the orientation lines of $\ro_y$; i.e. the free abelian group generated by the two possible orientations of $\ro_y$ with the relation that their sum vanishes.    In the case of exact Lagrangian sections of cotangent bundles, such an orientation is equivalent to an orientation of the descending manifolds of the corresponding critical point, and it is well known that explicit signs in Morse theory require choices of orientations on such descending manifolds.  By working with orientation lines, we avoid making these choices, at the cost of being less explicit about the resulting pluses and minuses.  For a thorough comparison between signs in Morse and Floer theory, see \cite{abouzaid-plumbing}.

 Imposing an additional condition on the choices made in Section \ref{sec:at-level-objects}, we require that our chosen triangulation satisfy the following condition:
\begin{equation}
  \label{eq:chords_in_skeleton}
  \parbox{36em}{Every Hamiltonian time-$1$ chord between elements of $\Ob(\Wrap_{b}(M))$ and every intersection point between such a Lagrangian and $Q$ is contained in the $3$-skeleton of $M$.}
\end{equation}
Note that while there are, a priori, infinitely many such chords, only finitely many of them intersect any given compact subset, which implies the existence of such a triangulation of $M$.  Given a chord $y$ with endpoints on $K$ and $L$ (and letting $E_{b}$ denote the vector bundle fixed in Section \ref{sec:at-level-objects}), we define a \emph{relative $\Pin$ structure} on $\Lambda_{y}$ to be a $\Pin$ structure on $y^{*}(E_{b}) \oplus \Lambda_{y} $ whose restriction to an end of the chord agrees with the relative $\Pin$ structure on either $K$ or $L$.  By thinking of an intersection point between $L$ and $Q$ as a time-$1$ chord for the zero Hamiltonian, we also associate a path $\Lambda_{q}$ to such a point, on which a  relative $\Pin$ structure may be chosen.

We now consider the following situation: $L_{0}, \ldots, L_{d}$ is a collection of objects of $ \Wrap_{b}(M) $ and  $\vx = ( x_1,  \ldots, x_d  )$ a sequence of chords with $x_k$ starting at $L_{k-1}$ and ending at $L_{k}$.  In addition, we have either a chord $x_0 \in \Chord(L_0,L_d)$ or intersection points $q_0$ and $ q_d$  of $L_0$ and $L_d$ with $Q$.  We shall only need to know the following Lemma which is a minor generalisations of the standard result required to define operations in Lagrangian Floer theory away from characteristic $2$.   In our notation below, $\lambda$ stands for the top exterior power of the tangent space of the appropriate manifold.
\begin{lem} \label{lem:orienting_mod_spaces}
Choices of relative $\Pin$ structures on the Lagrangians $L_{k}$ and on the paths $\Lambda_{x_k}$ for each element of $\vx$ and for $x_0$ induce a canonical up to homotopy isomorphism
\begin{equation}
\label{eq:orientation_higher_product}\lambda( \Disc_{d}(x_0;\vx) )  \otimes   \ro_{x_d}  \otimes \cdots \otimes \ro_{x_1}  \cong  \lambda(\Disc_{d}) \otimes \ro_{x_0}
\end{equation}
whenever the moduli space is regular.  The additional choices of  relative $\Pin$ structures on  $\Lambda_{q_0}$ and  $\Lambda_{q_d}$ induce an isomorphism
\begin{equation}
   \label{eq:orientation_higher_maps} \lambda( \Pil_{d}(q_0 , \vx, q_d) )   \otimes  \ro_{q_d} \otimes   \ro_{x_d}  \otimes \cdots \otimes \ro_{x_1}\cong  \lambda(\Pil_{d}) \otimes \ro_{q_0}.
\end{equation}
\end{lem}
\begin{proof}[Sketch of Proof:]
   Equation (12.8) of \cite{seidel-book} is the special case where the background class $b$ vanishes and Proposition 4.2.2 of \cite{WW} (see, in particular  Equation (38)) covers the relatively $\Spin$ case. We shall therefore limit ourselves to a discussion of the relevance of $\Pin$ structures to sign considerations; we shall see that they are introduced in order to resolve an ambiguity coming from the existence of homotopically distinct isotopies between a fixed pair of Fredholm operators. 

Given a curve  $u \in \Pil_{d}(q_0 , \vx, q_d) $ with source $S$ we have a canonical up to homotopy isomorphism
\begin{equation} \label{eq:super_regular_iso}
  \lambda(\Pil_{d}(q_0 , \vx, q_d))  \cong \lambda( \Pil_{d}) \otimes \det(D_{u}),
\end{equation}
where $\det(D_{u})$ is the determinant line of $D_{u}$, defined as
\begin{equation}
 \det(D_{u})= \lambda^{\vee}( \coker(D_{u}))  \otimes \lambda( \ker(D_{u})) .
\end{equation}

The next step is to deform $u=u_{0}$ through a family of smooth curves $\{ u_{r}\}_{r \in [0,1]}$ with the same asymptotic and boundary conditions to a curve $u_{1}$ mapping $\partial S$ to the $3$-skeleton of $M$.  Because any homotopy of homotopies can be deformed so that the boundary still maps to the $3$-skeleton of $M$, the construction we give below does not depend on the choice of $\{u_{r}\}$ (see the analogous argument in Chapter 8 of \cite{FOOO}).  From the homotopy, we obtain an isomorphism
\begin{equation} \label{eq:homotopy_to_skeleton}
   \det(D_{u}) \cong    \det(D_{u_1}) .
\end{equation}

The asymptotic conditions satisfied by $D_{u_1}$ near the positive punctures imply that we may glue to it the operators $D_{q_d}, D_{x_d}, \ldots, D_{x_1}$ and obtain an operator on a disc with one negative puncture corresponding to $q_0$ which we denote:
\begin{equation}
  D_{u_1} \#  D_{q_d} \# D_{x_d} \# \cdots \# D_{x_1}.
\end{equation}
Two salient fact about this operator will be relevant to us:  First, gluing kernel and cokernel elements gives a natural identification
\begin{equation} \label{eq:gluing_isomorphism}
 \det\left(  D_{u_1} \#  D_{q_d} \# D_{x_d} \# \cdots \# D_{x_1} \right)  \cong  \det(D_{u_1}) \otimes  \ro_{q_d} \otimes   \ro_{x_d}  \otimes \cdots \otimes \ro_{x_1},
\end{equation}
as discussed, e.g. in Section (11c) of \cite{seidel-book}.  Next, the fact that all Lagrangians are \emph{graded} implies that the  boundary condition of this glued operator are isotopic to those of the operator $D_{x_0}$ through an isotopy which is constant near the end.  A \emph{choice} of isotopy defines an isotopy of Fredholm operators on the disc, and hence using Equation \eqref{eq:gluing_isomorphism}, an isomorphism
\begin{equation} \label{eq:iso_gluing_dbar_operator}
 \det(D_{u_1}) \otimes  \ro_{q_d} \otimes   \ro_{x_d}  \otimes \cdots \otimes \ro_{x_1}   \cong \ro_{x_0},
\end{equation}
which, together with Equations \eqref{eq:super_regular_iso} and \eqref{eq:homotopy_to_skeleton}, gives an isomorphism between the right spaces.

Note that the choice of isotopy is parametrised, up to homotopy, by the fundamental group of the space of Lagrangian paths between two fixed points, which is a group of order $2$ in dimensions greater than $2$; the two different choices give different signs in the above isomorphism and we must be able to make a coherent choice in order to obtain consistent operations away from characteristic $2$ (see, e.g. Lemma 11.17 of \cite{seidel-book}).  The choice of relative $\Pin$ structures enters only in this choice of homotopy.

Indeed, Definition \ref{def:brane} required that we choose a $\Pin$ structure on the vector bundle $TL|L_{[3]} \oplus E_{b}|L_{[3]}$ over the $3$-skeleton of every object $L$ of $\Wrap(M)$.  This choice determines a $\Pin$ structure  on (1) the direct sum of the boundary condition of the operator $D_{x_0}$ with the trivial vector bundle with fibre  $E_{b} | q_0$ and (2) the direct sum of the boundary condition for $ D_{u_1} \#  D_{q_d} \# D_{x_d} \# \cdots \# D_{x_1}$ with the vector bundle
\begin{equation} \label{eq:vector_bundle_boundary}
 u_{1}^{*}( E_{b}) |\partial S \#  E_{b}|q_d \# (x_{d} \circ t)^{*}E_{b}| \partial T \# \cdots \# (x_{1} \circ t)^{*}E_{b}| \partial T
\end{equation}
over the boundary of the once-punctured disc obtained by gluing the pull-back of $E_{b}$ to $\partial S$ with the pullbacks of $E_{b}$ to the one punctured disc $T$ under the compositions of the chords  $q_d, x_d, \ldots, x_1$ with the projection $t \co T \to [0,1]$  fixed in Equation \eqref{eq:project_disc_strip}.   Note that the pullback of $E_{b}$ by $u_1$, and by $ x_{k} \circ t $, extends the  vector bundle \eqref{eq:vector_bundle_boundary} to the entire Riemann surface, hence induces a canonical up to homotopy isomorphism between this glued vector bundle and the trivial vector bundle with fibre $ E_{b}| q_{0} $.

With this in mind, we now fix the path between the glued operator and $D_{x_0}$ to be the unique one, up to homotopy, such that direct sum of the boundary conditions with the trivial bundle $ E_{b}|q_0 $ admits a $\Pin$ structure which restricts, at both ends, to the one fixed above.
\end{proof}
\begin{rem} \label{rem:trivialise_R_action}
Whenever $d=1$ in Equation \eqref{eq:orientation_higher_product}  (or $0$ in Equation \eqref{eq:orientation_higher_maps}), it is useful to think of  $ \Disc_1$ or $\Pil_0$  as a manifold of dimension $-1$ whose tangent space is ``generated'' by translation on the strip.
\end{rem}

Let us now consider the special case where $M$ is the cotangent bundle of $Q$, and the Lagrangians $L_k$ agree with a cotangent fibre $\Tq$.  We choose a vector bundle $E_{b}$ on a $3$-skeleton for $\TQ$ with Stiefel-Whitney class $b \in H^{*}(\TQ, \bZ_{2})$, and fix a submanifold $N_{b} \subset \TQ $ Poincar\'e dual to $b$, and which is disjoint from $T^{*}_{q} Q$ as well as from all time-$1$ Hamiltonian chords of $H$ with endpoints on this cotangent fibre.

Since it is contractible, the cotangent fibre $\Tq$ is an object of the wrapped Fukaya category of $\TQ$ for all possible twistings, and we would like to compare the isomorphisms (\ref{eq:orientation_higher_product}) for an arbitrary class $b$, with the one coming from the trivial class.  In order to do this, we choose a $\Pin$ structure on the Lagrangian path $\Lambda_{x}$ associated to any chord $ x \in \Chord(\Tq, \Tq)$.  Then, noting that the restriction of $E_{b}$ to the complement of $N_b$ in $\TQ_{[3]}  $ admits a $\Pin$ structure, we choose such a structure, and equip
\begin{equation}
  \Lambda_{x_k} \oplus x^{*}(E_{b})
\end{equation}
with the associated direct sum of $\Pin$ structures.

\begin{lem} \label{lem:twisted_to_untwisted_sign_difference}
Given a curve $u \in \Disc_{d}(x_0;\vx)   $,  the two isomorphisms 
\begin{equation} \label{eq:gluing_iso_dbar_disc}
 \det(D_{u}) \otimes  \ro_{x_d}  \otimes \cdots \otimes \ro_{x_1}   \cong \ro_{x_0},
\end{equation}
obtained from the sign conventions associated to $\Tq$ either as a $\Pin$ submanifold of $\TQ$, or as a relatively $\Pin$ submanifold with respect to the background class $b \in H^{*}(\TQ, \bZ_{2}) $ differ by $(-1)^{N_{b} \cdot u}$.  
\end{lem}
\begin{proof}
Fix an isotopy of operators which induces the isomorphism (\ref{eq:gluing_iso_dbar_disc}) in the untwisted case, i.e. such that the associated family of Lagrangian boundary conditions admits a $\Pin$ structure restricting at the two endpoints of the isotopy to the $\Pin$ structure on $\Lambda_{x_0}$, and on the glued boundary conditions. By definition, this given isotopy induces the desired isomorphism for the twisted case if and only the direct sum with the pullback of $E_{b}$ admits a $\Pin$ structure, i.e. if and only if the pullback of $E_{b}$ under $u_{1}$ is $\Pin$.  The obstruction is precisely the intersection number of $u_{1}$ with $N_{b}$, which, because $u$ and $u_1$ have the same asymptotic and boundary conditions, agrees with the intersection number of $u$ with $N_{b}$.
\end{proof}

\subsection{Signed operations} \label{sec:signs}

To obtain operations from Equations \eqref{eq:orientation_higher_product} and \eqref{eq:orientation_higher_maps}, one must choose orientations on the abstract moduli space $\Disc_{d} $ and $\Pil_{d} $,  or in the exceptional cases,  a trivialisation of the $\bR$ action on the moduli spaces of strips.  The orientations on $ \Disc_{d} $ is fixed by identifying with the configuration space of $d-2$ ordered points on an interval, which correspond to the marked points $\xi^{3}, \ldots, \xi^{d}$; this also fixes an orientation on $ \Pil_{d} $  via the isomorphism of Equation \eqref{eq:half_discs_are_discs}.  In the case of a strip, we shall use $\partial_{s}$ as the generator of the action of $\bR$.      For simplicity, we shall also fix trivalisations of $\ro_{q}$ for every intersection point between a Lagrangian $L \in \Ob(\Wrap_{b}(M))$ and $Q$.

We shall use the following conventions in determining signs:  Every orientation on a manifold $X$ induces an orientation on its boundary via the canonical isomorphism
\begin{equation} \label{eq:orient_boundary}
  \lambda(X) \cong \nu \otimes \lambda(\partial X)
\end{equation}
where $\nu$ is the normal bundle of $\partial X$ which is canonically trivialised by the outwards normal vector.  Whenever $X$ is a moduli space, say for specificity $ \Disc_{d}(x_0;\vx)   $,  Lemma \ref{lem:orienting_mod_spaces} gives an orientation \emph{relative} the orientations lines associated to the inputs and outputs (and an orientation of the abstract moduli space).  In all the cases we consider, the boundary of the moduli space is a product, and hence inherits two relative orientations, one coming from Equation \eqref{eq:orient_boundary}, the other from taking the product of the two relative orientations coming from Lemma \ref{lem:orienting_mod_spaces}.  The discrepancy between these two signs must be computed.  We illustrate the computation by providing the proof of some sign-related results.  The first is a Lemma from Section \ref{sec:at-level-objects}:

\begin{proof}[Proof of Lemma \ref{lem:product_orientations_strips}]
By rearranging the terms in Equation \eqref{eq:orientation_higher_product}, we obtain an isomorphism
\begin{equation} \label{eq:orientation_isomorphism_discs} \lambda(\Pilbar(q_i, q_j)) \cong \langle \partial_s \rangle  \otimes \ro_{q_i} \otimes \ro^{\vee}_{q_j} . \end{equation}
Taking the tensor product of these isomorphisms for the pairs $(q_i,q_k)$ and $(q_k, q_j)$, we obtain
\begin{equation*} \lambda(\Pilbar(q_i, q_k)) \otimes   \lambda(\Pilbar(q_k, q_j)) \cong \langle \partial_s\rangle \otimes \ro_{q_i} \otimes \ro^{\vee}_{q_k} \otimes  \langle \partial_s\rangle \otimes \ro_{q_k} \otimes \ro^{\vee}_{q_j}. \end{equation*}
The key observation coming from gluing theory is that the first translation vector field corresponds to a vector pointing outwards along this boundary stratum of $\Pilbar(q_i, q_j)$ (see Section (12f) of \cite{seidel-book}). If we identify it with the normal vector $\nu$, reorder the terms, and cancel the copy of $\ro^{\vee}_{q_k} \otimes \ro_{q_k}$  we arrive at
\begin{equation}  \lambda(\Pilbar(q_i, q_k)) \otimes   \lambda(\Pilbar(q_k, q_j)) \cong \nu \otimes    \langle \partial_s\rangle \otimes \ro_{q_i} \otimes \ro^{\vee}_{q_j}. \end{equation}
with a Koszul sign equal to
\begin{equation}(-1)^{|q_k| + |q_i|}  . \end{equation}  
\end{proof}

Next, we prove a Lemma from Section \ref{sec:construction-functor}

\begin{proof}[Proof of Lemma \ref{lem:orientaiton_error_moduli_half_disc_1_puncture}]
Since $\Pil_{1}$ is rigid, Equation \eqref{eq:orientation_higher_maps} specialises to
\begin{equation*}
\lambda( \Pil_{1}(q_0 , x, q_1) )  \cong   \ro_{q_0} \otimes  \ro_{q_1}^{\vee}   \otimes \ro^{\vee}_{x}.
\end{equation*}
If we take the tensor product of Equation \eqref{eq:orientation_isomorphism_discs} with the corresponding isomorphism for $  \Pil_{1}( q_0 , x_0 , q_1 )  $, we obtain
\begin{equation*}
  \lambda( \Pil_{1}(q_0 , x_0 , q_1) ) \otimes \lambda (   \Disc(x_0,x))  \cong   \ro_{q_0}   \otimes \ro^{\vee}_{x_0}  \otimes  \ro^{\vee}_{q_1} \otimes \langle \partial_s \rangle  \otimes \ro_{x_0} \otimes \ro^{\vee}_{x}
\end{equation*}
The translation vector $\partial_s  $  gives rise to an outward pointing normal vector upon gluing, so the only difference between this and the boundary orientation comes from  Koszul signs.  By assumption, $  \Disc(x_0,x)  $ is rigid, so there is no sign associated to moving $ \ro^{\vee}_{q_1}  $ past $ \langle \partial_s \rangle  \otimes \ro_{x_0} \otimes \ro^{\vee}_{x}  $.  The total sign is therefore   $(-1)^{|x_0| + |q_0| }$  coming from moving $ \langle \partial_s \rangle$  to be the first factor.

Next, we consider the situation where a strip breaks at the intersection point between $L_1$ and $Q$:
\begin{equation*}
  \lambda( \Pil_{1}(q_0 , x , q'_1) ) \otimes \lambda (   \Disc(q'_1, q_1)  )  \cong   \ro_{q_0}  \otimes \ro^{\vee}_{x}   \otimes  \ro^{\vee}_{q'_1} \otimes \langle \partial_s \rangle  \otimes \ro_{q'_1} \otimes \ro^{\vee}_{q_1}
\end{equation*}
After gluing, the vector $ \partial_s $ gives an inward pointing normal vector, so the total sign difference between the two orientations is $(-1)^{ 1+ |q_0| + |x| + | q'_1|  } $.  Finally, when breaking occurs at the intersection point between $L_0$ and $Q$, the reader may easily check that the two relative orientations agree.
\end{proof}

Finally, we prove a result from Section \ref{sec:construction-functor}:
\begin{proof}[Proof of Lemma \ref{lem:product_orientation_half_discs}]
When breaking takes place on the outgoing segment, Equation \eqref{eq:orientation_higher_maps} implies that the product orientation is given by an isomorphism 
\begin{multline*}
  \lambda(\Pil_{d_1}(q_0,\vx[1],q'_{d_1}))  \otimes   \lambda(\Pil_{d_2}(q'_{d_1},\vx[2],q_d))   \cong \lambda(\Pil_{d_1})  \otimes \ro_{q_0} \otimes \ro_{x_1}^{\vee} \otimes \cdots \otimes  \ro_{x_{d_1}}^{\vee} \otimes \ro_{q'_{d_1}}^{\vee} \\
\otimes \lambda(\Pil_{d_2})  \otimes \ro_{q'_{d_1}} \otimes \ro_{x_{d_1+1}}^{\vee} \otimes \cdots \otimes  \ro_{x_d}^{\vee} \otimes \ro_{q_d}^{\vee} .
\end{multline*}
As the dimension of $\Pil_{d_2}  $ is $d_2 +1$, the Koszul sign introduced by moving $\lambda( \Pilbar_{d_2}) $ to be adjacent to $ \lambda( \Pilbar_{d_1}) $ is $ (d_2+1)\left( |q'_{d_1}| + | q_0 | +  \sum_{i=1}^{d_1} | x_{i}| \right) $.  The additional term in Equation \eqref{eq:sign_flat} comes from the difference between the product orientation of $\Pil_{d_1} \times \Pil_{d_2}$ and the orientation of $\Pil_{d} $.
At the other type of boundary stratum, the product orientation is given by 
\begin{multline*} \lambda(\Pil_{d_1}(q_0,\vx[1],q_d))  \otimes   \lambda(\Disc_{d_2}(y;\vx[2])   \cong \lambda(\Pil_{d_1})  \otimes \ro_{q_0} \otimes \ro_{x_1}^{\vee} \otimes \cdots \otimes \ro_{x_k}^{\vee} \otimes \ro_{y}^{\vee}\otimes   \ro_{x_{k+d_2+ 1}}^{\vee} \otimes \cdots  \otimes  \ro_{x_d}^{\vee} \otimes \ro_{q_d}^{\vee}  \\ \otimes \lambda(\Disc_{d_2})   \otimes \ro_{y} \otimes  \ro_{x_{k+1}}^{\vee} \otimes \cdots \otimes  \ro_{x_{k+d_2} }^{\vee}. \end{multline*}
Since we're assuming that $\Disc(y,\vx[2])  $ is rigid, there is no sign associated to reordering the factors of the right hand side as
 \begin{equation*}
\lambda(\Pil_{d_1})  \otimes \ro_{q_0} \otimes \ro_{x_1}^{\vee} \otimes \cdots \otimes \ro_{x_k}^{\vee} \otimes \ro_{y}^{\vee}\otimes  \lambda(\Disc_{d_2})  \otimes \ro_{y} \otimes \ro_{x_{k+1}}^{\vee} \otimes  \cdots  \otimes  \ro_{x_d}^{\vee} \otimes \ro_{q_d}^{\vee}   .
\end{equation*}
Moving $ \lambda(\Disc_{d_2}) $ to be adjacent to $\lambda(\Pil_{d_1})  $ introduces a Koszul sign of $d_2 (d_2 +  |q_0| + \sum_{j=1}^{k+d_2} |x_k| ) $.  The difference with Equation \eqref{eq:sign_sharp} again comes from the orientation of the boundary of the abstract moduli space of discs.   
\end{proof}

\subsection{On the gluing theorem}
\label{app:manifold-with-boundary}
In this section we sketch the proof of Proposition \ref{prop:moduli_spaces_strips_manifold}.  This is a standard consequence of a gluing theorem, proved in slightly different ways by many different people, but we will describe a specific approach, expecting that expert readers will prefer to replace it by their own.  The method we explain in this appendix has the advantage of using only the most elementary version of the implicit function theorem in its proof, but the disadvantage of requiring a series of auxiliary geometric choices.    Our approach closely follows that of \cite{FOOO}.

First, the Gromov compactification of  $\Discbar(q_i, q_j)$ is by definition the union of $ \Disc(q_i, q_j) $ with the images of embeddings
\begin{equation} \Discbar(q_i, q_{i_1}) \times \Discbar(q_{i_1}, q_j) \to  \Discbar(q_i, q_j) \end{equation}
which satisfy \eqref{eq:compatible_boundary_strata_inclusions}.  An application of the Sard-Smale theorem in the appropriate Banach space implies that  $ \Disc(q_i, q_j) $ is a smooth manifold of the appropriate dimension for generic choices of perturbation data.  An index computation implies that the difference between the dimension of $ \Disc(q_i, q_j) $  and 
\begin{equation*} \Disc(q_i, q_{i_1}) \times  \cdots \times \Discbar(q_{i_r}, q_{j})  \end{equation*} 
is $r$ (the number of factor minus $1$).  Let $\vu = (u^0, \ldots, u^r)$ denote a point in such a stratum.  We must prove that we may choose a neighbourhood $U$ of $\vu$ in this product of lower dimensional manifolds so that there is local homeomorphism (the result of a gluing theorem)
\begin{equation}\G \co  U \times [0,\epsilon)^{r} \to  \Discbar(q_{i}, q_{j}) . \end{equation}

In practice one uses the stratification on $[0,\epsilon)^{k}$ and constructs the above map separately for each stratum, checking ultimately that the resulting map is continuous; we shall focus on the top stratum in which none of the gluing coordinates vanish.  To construct these gluing map, we assume without loss of generality that each map $u_\ell$ has non-vanishing derivative at the origin (this may be achieved by pre-composing $u_\ell$ by a translation), and set $z_0^{\ell} = 0$.  In addition, writing  $d_{\ell}$  for the dimension of  $\Disc(q_\ell, q_{\ell+1})$, we choose a sequence of points $(z^{\ell}_1, \ldots, z^{\ell}_{d_{\ell}}) \in Z^{d_{\ell}}$ which are regular points for $u_{\ell}$ and such that the derivative of the evaluation map at these points
\begin{equation*}\ev^{\ell} = (\ev_{z^\ell_1}, \ldots , \ev_{z^\ell_{d_{\ell}}} ) \co  \Disc(q_\ell, q_{\ell+1}) \to M^{d_\ell} \end{equation*}
is an isomorphism at $u_{\ell}$.  We fix a product of hypersurfaces 
\begin{equation*} N_{\ell} = N_{1,\ell} \times \cdots \times N_{d_\ell+1,\ell} \subset  M^{d_\ell+1} \end{equation*} 
 which is transverse to the image of 
\begin{equation*}   D_{\vu} \ev^{\ell}   \oplus  D_{z_0^{\ell}} u^{\ell}|\partial_t \end{equation*} 
and such that the inverse image of each factor under $u^{\ell}$ is transverse to the lines obtained by fixing the $t$-coordinate on the strip.

Given a sequence  $\vv = (v^0, \ldots, v^r)$ of holomorphic maps sufficiently close to $\vu$, there is a unique parametrisation  $v^{\ell}$ which $C^{\infty}$ close to $u^{\ell}$ and such that $0 \in (v^{\ell})^{-1} (N_{d_\ell+1,\ell})$.  Having fixed this parametrisation, the inverse image under $v^{\ell}$ of the remaining factors of $N_{\ell}$ determines a unique sequence of points $ (z^{\ell}_1(v^{\ell}) , \ldots, z^{\ell}_{d_{\ell}}(v^{\ell})) $ which are $C^0$ close to the sequence   $(z^{\ell}_1, \ldots, z^{\ell}_{d_{\ell}})$ and such that the $t$-coordinates agree.  In particular, fixing a sufficiently small neighbourhood $U_{\ell}$ of $v^{\ell}$ in  $\Disc(q_\ell, q_{\ell+1})$, we obtain an embedding
\begin{equation}  \label{eq:embedding_moduli_space_marked_points} U_{\ell} \to Z^{d_{\ell}} .\end{equation}

An elementary application of the (finite dimensional) implicit function theorem implies:
\begin{lem} The composition of \eqref{eq:embedding_moduli_space_marked_points} with the projection to the $s$-coordinate defines a local diffeomorphism
\begin{equation} U_{\ell} \to  \bR^{d_{\ell}} .\end{equation} \noproof
\end{lem}
To summarise this discussion, we have chosen a product of hypersurfaces such that the moduli space of holomorphic discs near $u^{\ell}$ is parametrised by the $s$-coordinates of the unique inverse image of $  N_{k,\ell} $  whose $t$-coordinate agrees with $t$-coordinate of the original marked point $z^{\ell}_k$.

The choice of  $N_{\ell}$ determines a family of right inverses to the linearisation $\dbar$ operator
\begin{equation} C^{\infty}(u_{\ell}^{*} T M) \to C^{\infty}(u_{\ell}^{*} T M \otimes \Omega^{0,1}), \end{equation}
Indeed, this operator becomes an isomorphism when restricted to those sections whose value at $z^{\ell}_r$ lies in $TN_{r,\ell}$.   The inverse of this isomorphism gives a right inverse to the linearisation of the $\dbar$ operator in any reasonable Sobolev completion of the space of smooth sections.  To be precise, one must work with functions satisfying an exponential decay condition along the ends of the strip, which is reflected in weighted Sobolev spaces when passing to completions.

We are now ready to proceed with gluing holomorphic curves.  First, we define a pre-gluing map
\begin{align*} \preG \co   U \times (0,\epsilon)^{r} & \to C^{\infty}(Z, M)  \\
(\vv , \vepsilon)  = (v^0, \ldots, v^r, \epsilon_1, \ldots, \epsilon_r) & \mapsto v^0 \#_{\epsilon_1} v^1 \#_{\epsilon_2 } \cdots \#_{\epsilon_{r}} v^r   
 \end{align*}
obtained by removing the half-infinite strips 
\begin{equation*} [\frac{2}{\epsilon_\ell}, +\infty) \times [-1,1] \textrm{ and } (-\infty, \frac{-2}{\epsilon_\ell} ] \times [-1,1] \end{equation*}
respectively from the domains of the curves $(v^{\ell-1},v^{\ell})$ and gluing the corresponding maps along the finite strips
\begin{equation*} [\frac{1}{\epsilon_\ell}, \frac{2}{\epsilon_\ell}] \times [-1,1] \textrm{ and } [\frac{-2}{\epsilon_\ell}, \frac{-1}{\epsilon_\ell} ] \times [-1,1] \end{equation*}
using a cutoff function.  This makes sense because these curves exponentially converge along the appropriate end to the same point in $M$.   We assume that $\epsilon$ is small enough that the regions along which we glue the curves are disjoint from all marked points on the domains of $u_{\ell-1}$ and $u_{\ell}$.

As usual, there is an $\bR$ ambiguity in the parametrisation of $\preG(\vv,\vepsilon)$.  We resolve this ambiguity by making the origin of the domain of $\preG(\vv,\vepsilon)$ agree with the marked point coming from the origin of the domain of $v^r$. In other words, we start with the domain of $v^r$, and glue the domains of the remaining curves to it.  In particular, the pre-glued curve is equipped with $r-1$ marked points
\begin{equation*}  (z^0_0(\vepsilon), \ldots   ,  z^{r-1}_0(\vepsilon) )  \end{equation*}
coming from the origins of the domains of the curves $v^{\ell}$ (in particular, these marked points all have vanishing $t$-coordinate, and their $s$-coordinate is a sum of terms of the form $\frac{3}{\epsilon_\ell}$).  Moreover, there are $\sum_{\ell=1}^{r} d_{\ell} $ marked points
\begin{equation*}  (z^0_1(\vv, \vepsilon), \ldots   , z^0_{d_{0}} (\vv, \vepsilon), z^1_1(\vv, \vepsilon),\ldots, z^1_{d_{1}} (\vv, \vepsilon), z^2_1(\vv, \vepsilon),  \ldots , z^{r}_{d_{r}}(\vv, \vepsilon) )  \end{equation*}
 coming from the marked points on each curve $v^{\ell}$. The $s$ coordinate of of  $z^{\ell}_{k}(\vv, \vepsilon)$  is simply the sum  $z^{\ell}_0(\vepsilon) +  z^{\ell}_{k}(v_{\ell})$. Letting $d$ stand for the dimension of $\Disc(q_i,q_j)$, we obtain a sequence of $d$ marked points
\begin{equation}  (z_1(\vv, \vepsilon), \ldots, z_{d}(\vv, \vepsilon))  \end{equation}
on the domain of $\preG(\vv,\vepsilon)$ with the convention that the last $r$ marked points correspond to the origins of the domains of $\vv$, and we set $z_0(\vv,\vepsilon)$ to be the origin as before.    Note that at each such marked point the image of $\preG(\vv,\vepsilon)$ intersects a distinguished hypersurface $N_{k}$ transversely; these hypersurfaces are the ones we chose earlier in order to rigidify each curve $v^{\ell}$.

As $\preG(\vv,\vepsilon)$ is not holomorphic, there is no natural linearisation of the $\dbar$ operator acting on the tangent to the  space of smooth maps at $\preG(\vv,\vepsilon)$.  To obtain such a linearisation, we choose for each curve $v^{\ell}$, a family of metrics on $M$ parametrised by $Z$ which is constant (and independent of $\ell$ away from a compact set) and such that $N_{\ell,k}$ is totally geodesic with respect to the metric corresponding to the marked point $z^{\ell}_k(\vv)$.  Note that for $\epsilon$ sufficiently small, this choice naturally induces a choice of metrics in $M$ parametrised by the domain of  $\preG(\vv,\vepsilon)$ such that $N_{k}$ is totally geodesic with respect to the metric corresponding to $z_k(\vv,\vepsilon)$.

Let us write
\begin{equation} C^{\infty}_{\vN}(\preG(\vv,\vepsilon)^{*} TM)   \end{equation}
for those sections of $ \preG(\vv,\vepsilon)^{*} TM $ whose value at the $k$\th marked point lies in the tangent space to $N_k$.   Again, one must impose an appropriate decay condition at the ends in order to obtain a well-behaved elliptic problem.
\begin{lem}
For $\epsilon$ sufficiently small, the linearisation of the $\dbar$ operator becomes an isomorphism when restricted to $C^{\infty}_{\vN}(\preG(\vv,\vepsilon)^{*} TM)$.  Moreover, in an appropriate Sobolev completion of the space of smooth functions, the following quantities are uniformly bounded with respect to $\vepsilon$:
\begin{enumerate}
\item The norm of $\dbar(\preG(\vv,\vepsilon)  ) $.
\item The norm of the unique right inverse to the linearisation of the $\dbar$ operator  whose image is  $C^{\infty}_{\vN}(\preG(\vv,\vepsilon)^{*} TM)$.
\item The difference between linearisation of the $\dbar$ operator at $\preG(\vv,\vepsilon)$ and at points obtained by exponentiating vector fields of a fixed norm.
\end{enumerate}
\noproof
\end{lem}
\begin{rem}
The most subtle part of the construction is the choice of a Sobolev completion in which all these properties hold at once.  Following \cite{FOOO}, we would prove this result by choosing a family of weights depending on the parameter $\vepsilon$.  The properties are listed in what we believe is their order of difficulty: the first is an essentially elementary estimate while the second can be proved by constructing a pre-gluing map at the level of tangent vector field following \cite{FOOO}.  In general, one cannot explicitly describe the result of pre-gluing tangent vectors in the image of the chosen right inverse along the component curves of $\vv$.  However, for our explicit choices of right inverses, one can easily check that the result is the Sobolev completion of $C^{\infty}_{\vN}(\preG(\vv,\vepsilon)^{*} TM)$. The last property above is a version of the quadratic inequality (see Appendix B of \cite{abouzaid-exotic} for a proof in a related setting).
\end{rem}

The previous Lemma is the necessary analytic input to apply the implicit function theorem.  In this case, the consequence is rather easy to state:
\begin{cor}
There exists a unique vector in $C^{\infty}_{\vN}(\preG(\vv,\vepsilon)^{*} TM)$, which is bounded with respect to the preferred Sobolev norm, and whose image under the exponential map is an honest solution to the $\dbar$ equation.  \noproof
\end{cor}

The gluing map
\begin{equation} \G \co U \times  (0,\epsilon)^{r} \to \Disc(q_i,q_j)\end{equation}
is defined by exponentiating this vector field at every curve obtained by pre-gluing.  The proof that this map  extends continuously upon letting some of the gluing parameters go to $0$ is a direct consequence of Gromov compactness, and is omitted.   It remains therefore to explain why this map is a homeomorphism onto the intersection of a neighbourhood of $\vu$ in $\Discbar(q_i,q_j)$ with $\Disc(q_i,q_j)$.  

To prove this, we first observe that, as a consequence of Gromov compactness, there is a neighbourhood $W$ of $\vu$ in $\Disc(q_i,q_j)$ consisting of curves $w$ which, for each integer $k$ admit a unique point $z_k(w)$ close to $z_k(\vu, \vepsilon )$ for some sequence $\vepsilon$, such that the $t$-coordinates of $z_k(w)$ and $z_k(\vu, \vepsilon )$ agree, and
\begin{equation*}\parbox{36em}{ $N_k$ intersects the image of $w$ transversely at  $z_k(w)$. } \end{equation*}
In particular, taking the $s$-coordinate of each point $z_k$ defines a map
\begin{equation*} R \co  W \to \bR^{d} .\end{equation*}

As the right inverse to the $\dbar$ operator consists of vector fields which, at $z_{k}(\vv, \vepsilon)$ point in the direction of $N_k$, and as $N_k$ is totally geodesic with respect to the chosen metric, we conclude that $\G(\vv, \vepsilon)$ intersects $N_k$ at  $z_{k}(\vv, \vepsilon)$.  This implies:
\begin{lem}
The map $R$ is a right inverse to $\G$. More formally, we have a commutative diagram
\begin{equation*} \xymatrix{ U \times  (0,\epsilon)^k \ar[rr]^{\G} \ar[dr] & & W \ar[dl]^{R} \\
& \bR^{d}  & 
  }  \end{equation*}
where the map  $U \times  (0,\epsilon)^k \to \bR^{d} $ is the codimension $0$ embedding which takes $(\vv, \vepsilon)$ to the $s$-coordinates of $z_{k}(\vv, \vepsilon)$. \noproof
\end{lem}
\begin{cor}
The gluing map $\G$ is injective. \noproof
\end{cor}

The existence of this a priori right inverse to the gluing map also helps in proving surjectivity.  Indeed, if the images of $w \in W$ and $\G(\vv,\vepsilon)$ under $R$ agree, then $w$ can be expressed as the exponential of a vector field along $\preG(\vv,\vepsilon)$ lying in  $C^{\infty}_{\vN}(\preG(\vv,\vepsilon)^{*} TM)$.  By the uniqueness part of the implicit function theorem, $w$ and $\G(\vv,\vepsilon)$ must agree if we can prove that this vector field is bounded in the appropriate Sobolev norm.  For $w$ sufficiently close to $\vu$, this follows from Gromov compactness, and a standard estimate on the energy of long strips (see e.g. Section 5d of   \cite{abouzaid-exotic}).  This completes our sketch of the proof of Proposition \ref{prop:moduli_spaces_strips_manifold}.


\begin{bibdiv}
\begin{biblist}
\bib{AS0}{article}{
   author={Abbondandolo, Alberto},
   author={Schwarz, Matthias},
   title={On the Floer homology of cotangent bundles},
   journal={Comm. Pure Appl. Math.},
   volume={59},
   date={2006},
   number={2},
   pages={254--316},
   issn={0010-3640},
   review={\MR{2190223 (2006m:53137)}},
}
	

\bib{AS}{article}{
author={ Alberto Abbondandolo},
author={  Matthias Schwarz},
title={Floer homology of cotangent bundles and the loop product},
eprint={arXiv:0810.1995},
}

\bib{APS}{article}{
   author={Abbondandolo, Alberto},
   author={Portaluri, Alessandro},
   author={Schwarz, Matthias},
   title={The homology of path spaces and Floer homology with conormal
   boundary conditions},
   journal={J. Fixed Point Theory Appl.},
   volume={4},
   date={2008},
   number={2},
   pages={263--293},
   issn={1661-7738},
   review={\MR{2465553}},
}

	

\bib{abouzaid-exotic}{article}{
    author={Abouzaid, Mohammed},
title={Framed bordism and Lagrangian embeddings of exotic spheres},
eprint={arXiv:0812.4781},
}

\bib{abouzaid-plumbing}{article}{
   author={Abouzaid, Mohammed},
title={A topological model for the Fukaya categories of plumbings },
eprint={arXiv:0904.1474},
status={to appear in Journal of Differential Geometry},
}

\bib{generate}{article}{
    author={Abouzaid, Mohammed},
title={A geometric criterion for generating the Fukaya category },
eprint={arXiv:1001.4593},
}

\bib{fibre-generates}{article}{
    author={Abouzaid, Mohammed},
title={Maslov $0$ nearby Lagrangians are homotopy equivalent},
eprint={arXiv:1005.0358},
}
\bib{abouzaid-seidel}{article}{
   author={Abouzaid, Mohammed},
   author={Seidel, Paul},
   title={An open string analogue of Viterbo functoriality},
   journal={Geom. Topol.},
   volume={14},
   date={2010},
   number={2},
   pages={627--718},
   issn={1465-3060},
   review={\MR{2602848}},
   doi={10.2140/gt.2010.14.627},
}


\bib{BC}{article}{
   author={Barraud, Jean-Fran{\c{c}}ois},
   author={Cornea, Octav},
   title={Lagrangian intersections and the Serre spectral sequence},
   journal={Ann. of Math. (2)},
   volume={166},
   date={2007},
   number={3},
   pages={657--722},
   issn={0003-486X},
   review={\MR{2373371 (2008j:53149)}},
}



\bib{bondal-kapranov}{article}{
   author={Bondal, A. I.},
   author={Kapranov, M. M.},
   title={Framed triangulated categories},
   language={Russian},
   journal={Mat. Sb.},
   volume={181},
   date={1990},
   number={5},
   pages={669--683},
   issn={0368-8666},
   translation={
      journal={Math. USSR-Sb.},
      volume={70},
      date={1991},
      number={1},
      pages={93--107},
      issn={0025-5734},
   },
   review={\MR{1055981 (91g:18010)}},
}

\bib{floer-lagrangian}{article}{
   author={Floer, Andreas},
   title={Morse theory for Lagrangian intersections},
   journal={J. Differential Geom.},
   volume={28},
   date={1988},
   number={3},
   pages={513--547},
   issn={0022-040X},
   review={\MR{965228 (90f:58058)}},
}


\bib{FHS}{article}{
    author={Floer, Andreas},
    author={Hofer, Helmut},
    author={Salamon, Dietmar},
     title={Transversality in elliptic Morse theory for the symplectic
            action},
   journal={Duke Math. J.},
    volume={80},
      date={1995},
    number={1},
     pages={251\ndash 292},
      issn={0012-7094},
    review={MR1360618 (96h:58024)},
}
\bib{FOOO}{book}{
   author={Fukaya, Kenji},
   author={Oh, Yong-Geun},
   author={Ohta, Hiroshi},
   author={Ono, Kaoru},
   title={Lagrangian intersection Floer theory: anomaly and obstruction.
   Part II},
   series={AMS/IP Studies in Advanced Mathematics},
   volume={46},
   publisher={American Mathematical Society},
   place={Providence, RI},
   date={2009},
   pages={i--xii and 397--805},
   isbn={978-0-8218-4837-1},
   review={\MR{2548482}},
}

\bib{FSS}{article}{
   author={Fukaya, K.},
   author={Seidel, P.},
   author={Smith, I.},
   title={The symplectic geometry of cotangent bundles from a categorical
   viewpoint},
   conference={
      title={Homological mirror symmetry},},
   book={
      series={Lecture Notes in Phys.},
      volume={757},
      publisher={Springer},
      place={Berlin},
   },
   date={2009},
   pages={1--26},
  review={\MR{2596633}},
}

\bib{kragh}{article}{
author={Kragh, Thomas},
eprint={ arXiv:0712.2533},
title={The Viterbo Transfer as a Map of Spectra },
}

\bib{massey}{book}{
   author={Massey, William S.},
   title={A basic course in algebraic topology},
   series={Graduate Texts in Mathematics},
   volume={127},
   publisher={Springer-Verlag},
   place={New York},
   date={1991},
   pages={xvi+428},
   isbn={0-387-97430-X},
   review={\MR{1095046 (92c:55001)}},
}


\bib{MS}{book}{
   author={McDuff, Dusa},
   author={Salamon, Dietmar},
   title={$J$-holomorphic curves and symplectic topology},
   series={American Mathematical Society Colloquium Publications},
   volume={52},
   publisher={American Mathematical Society},
   place={Providence, RI},
   date={2004},
   pages={xii+669},
   isbn={0-8218-3485-1},
   review={\MR{2045629 (2004m:53154)}},
}

\bib{RS}{article}{
   author={Robbin, Joel W.},
   author={Salamon, Dietmar A.},
   title={Asymptotic behaviour of holomorphic strips},
   language={English, with English and French summaries},
   journal={Ann. Inst. H. Poincar\'e Anal. Non Lin\'eaire},
   volume={18},
   date={2001},
   number={5},
   pages={573--612},
   issn={0294-1449},
   review={\MR{1849689 (2002e:53135)}},
   doi={10.1016/S0294-1449(00)00066-4},
}



\bib{seidel-book}{book}{
   author={Seidel, Paul},
   title={Fukaya categories and Picard-Lefschetz theory},
   series={Zurich Lectures in Advanced Mathematics},
   publisher={European Mathematical Society (EMS), Z\"urich},
   date={2008},
   pages={viii+326},
   isbn={978-3-03719-063-0},
   review={\MR{2441780}},
}
\comment{
\bib{seidel-les}{article}{
   author={Seidel, Paul},
   title={A long exact sequence for symplectic Floer cohomology},
   journal={Topology},
   volume={42},
   date={2003},
   number={5},
   pages={1003--1063},
   issn={0040-9383},
   review={\MR{1978046 (2004d:53105)} ,}

}}
\bib{seidel-CP2}{article}{
   author={Seidel, Paul},
 title={A remark on the symplectic cohomology of cotangent bundles, after Thomas Kragh},
status ={unpublished note},
}

\bib{WW}{article}{
   author={Wehrheim, Katrin},
   author={Woodward, Chris T.},
 title={Orientations for pseudoholomorphic quilts},
status ={manuscript},
}


\end{biblist}
\end{bibdiv}
\end{document}